\documentclass[twocolumn,amsthm]{autart}%
\usepackage{todonotes}
\usepackage{color}
\usepackage{graphicx}
\usepackage{epsfig}
\usepackage{amsmath}
\usepackage{amssymb}
\usepackage{mathrsfs}
\usepackage{amsfonts}
\usepackage{dsfont}
\usepackage{graphicx}
\usepackage{mathrsfs}
\usepackage{wrapfig}
\usepackage{mathtools}

\usepackage{balance}
\usepackage[makeroom]{cancel}

\usepackage[round]{natbib}%
\providecommand{\U}[1]{\protect\rule{.1in}{.1in}}
\providecommand{\U}[1]{\protect\rule{.1in}{.1in}}

\theoremstyle{definition}

\newtheorem{definition}{Definition}
\newtheorem{remark}{Remark}
\newtheorem{problem}{Problem}

\theoremstyle{plain}
\newtheorem{proposition}{Proposition}

\newtheorem{theorem}{Theorem}
\newtheorem{lemma}{Lemma}

\begin{document}
\begin{frontmatter}
\title{Separation of Learning and Control for Cyber-Physical Systems\thanksref{footnoteinfo}} % Title, preferably not more
\thanks[footnoteinfo]{This work was supported by NSF under Grants CNS-2149520 and CMMI-2219761.}
\author[Paestum]{Andreas A. Malikopoulos}\ead{andreas@udel.edu}    % Add the % (ead) as shown
\address[Paestum]{Department of Mechanical Engineering, University of Delaware, 130 Academy Street, Newark, DE, 19716, USA}  % Please supply
%\address[Rome]{Division of Systems Engineering and Center for Information and Systems Engineering, Boston University, 15 Saint Mary's Street,
%Brookline, MA, 02446, USA}             % full addresses
%\address[Baiae]{The White House, Baiae}        % here.
                             % keyword list or with the
% help of the Automatica
% keyword wizard
\begin{abstract}                          % Abstract of not more than 200 words.
Most cyber-physical systems (CPS) encounter a large volume of data which is added to the system gradually in real time and not altogether in advance. In this paper, we provide a theoretical framework that yields optimal control strategies for such CPS at the intersection of control theory and learning. In the proposed framework, we use the actual CPS, i.e., the ``true" system that we seek to optimally control online, in parallel with a model of the CPS that is available. We then institute an information state for the system which does not depend on the control strategy. An important consequence of this independence is that for any given choice of a control strategy and a realization of the system’s variables until time $t,$ the information states at future times do not depend on the choice of the control strategy at time $t$ but only on the realization of the decision at time $t,$ and thus they are related to the concept of separation between estimation of the state and control. Namely, the future information states are separated from the choice of the current control strategy. Such control strategies are called separated control strategies. Hence, we can derive offline the optimal control strategy of the system with respect to the information state, which might not be precisely known due to model uncertainties or complexity of the system, and then use standard learning approaches  to learn the information state online while data are added gradually to the system in real time. We show that after the information state becomes known, the separated control strategy of the CPS model derived offline is optimal for the actual system. We illustrate the proposed framework in a dynamic system consisting of two subsystems with a delayed sharing information structure.
\end{abstract}
\end{frontmatter}

\section{Introduction} \label{sec:1}
\subsection{Motivation} \label{sec:1a}
Cyber-physical systems (CPS), in many instances, represent systems of subsystems with an informationally decentralized structure such as networked control systems, emerging mobility systems,  communication networks, digital twin, and internet of things.
Systems with informationally decentralized structures impose significant challenges compared to systems with centralized \textit{information structures}; see \citet{Schuppen2015}. 
The information structure in a system designates what information each subsystem knows about the status of the system and when. Several efforts on the characterization of information structures and their implications on optimality results have been reported in the literature over the years; see \citet{Witsenhausen1971a, Mahajan2012, Subramanian2020ApproximateIS}.
The information structure in a system stipulates the complexity, i.e., see \citet{papadimitriou1982complexity,tsitsiklis1985complexity}, of the optimal control problem and can lead to computational implications; see \citet{Papadimitriou:1985aa}. 
The latter depends on whether the system has a \textit{strictly classical} information structure or a  \textit{nonclassical} information structure. In classical information structures, all subsystems receive the same information  and have perfect recall; see \citet{Malikopoulos2016c}. If there is only one subsystem, then such information structures are called \textit{strictly classical} resulting in typical centralized stochastic control problems; see \citet{Kumar1986, Kushner1971}. In partially nested information structures, there are some subsystems who have a nonempty intersection of their information structures while they have perfect recall. Any information structure that is not classical, or partially nested, is called nonclassical.

In most CPS applications with nonclassical information structures there is  a large volume of data of a dynamic nature which is added to the system gradually in real time and  not altogether in advance.
As the volume of data  increases, the domain of the control strategies also increases, and thus it becomes challenging to search for an optimal strategy. Even if an optimal strategy is found, implementing such strategies with increasing domains is burdensome.
In such applications, we typically assume an ideal model of the system to derive optimal control strategies. Such  model-based control approaches cannot effectively facilitate optimal solutions with performance guarantees due to the discrepancy between the model and the actual CPS. On the other hand, traditional supervised learning approaches cannot always  facilitate robust solutions using data derived offline. By contrast, applying reinforcement learning approaches directly to the actual CPS might impose significant implications on safety and robust operation of the system. 

The  goal of this paper is to provide a theoretical framework that  aims at separating the control and learning tasks which allows us to combine offline model-based control with online learning approaches, and thus circumvent the challenges in deriving optimal strategies for CPS with nonclassical information structures. The framework can fit well in applications related to digital twins where a virtual representation of a real-world physical system  serves as the indistinguishable digital counterpart of it.

\subsection{Related Work} \label{sec:1b}
\subsubsection{Model-Based Control}
Most CPS represent systems of subsystems with nonclassical information structures imposing the following technical challenges \citep{Papadimitriou:1987:CMD:2875343.2875347}:
(a) the functional optimization problem of selecting the optimal strategy is not trivial as the class of strategies is infinite dimensional, and (b) the data increase with time causing significant implications on storage requirements and real-time implementation. These difficulties can be addressed  by finding sufficient statistics to compress the growing data without loss of optimality \citep{Striebel1965} using a conditional probability of the state of the system at time $t$ given all the data available up until time $t$. This conditional probability is called \textit{information state}, and it takes values in a time-invariant space. This information state can help us derive results for optimal control strategies in a time-invariant domain; \citet{Krishnamurthy2016}. 

One key property of such information states is that they do not depend on the control strategy of the system, and thus they are related to the concept of separation between estimation and control. An important consequence of this separation is that for any given choice of control strategies and a realization of the system's variables until time $t$, the information states at future times do not depend on the choice of the control strategy at time $t$ but only on the realization of the decision at time $t$; see \citet{Malikopoulos2021}. Thus, the future information states are separated from the choice of the current control strategy. The latter is necessary in order to formulate a classical dynamic program \citep{Howard,Bertsekas2017}, where at each step the optimization problem is to find the optimal decision for a given realization of the information state. 

Several optimality results using information states defined in time-invariant spaces have been reported in the literature for systems with nonclassical information structures; see \citet{Witsenhausen1971a,Varaiya:1978aa, Kurtaran:1979aa, Nayyar2011,wu2014theory,Gupta:2015aa, Dave2019a,Dave2020a}. There are three main approaches to address optimal control problems with a nonclassical information structure: (1) the \textit{person-by-person} approach, (2) the \textit{designer's} approach, and (3) the \textit{common information} approach.  
The person-by-person approach \citep{McGuire1972} aims to convert the problem into a centralized stochastic control problem from the point of view of each subsystem. Namely, we arbitrarily fix the strategies for all subsystems except for one, say subsystem $k \in \mathcal{K}$, $\mathcal{K}=\{1,\ldots,K\}$, $K\in\mathbb{N}$, and then, we derive the optimal strategy  for $k$ given the strategies for all other subsystems. We repeat this process for all subsystems until no subsystem can improve the performance of the system by unilaterally changing their strategy.  
The designer's approach was first introduced by  \citet{Witsenhausen1973}, as a standard form for sequential stochastic control with a nonclassical information structure, and extended later by \citet{mahajan2008sequential}. The designer's approach transforms the problem into a centralized, open-loop planning problem where the objective is to derive the optimal control strategy of the system before the system starts evolving. Thus, no data are observed by the designer, and thus this approach leads to a dynamic programming decomposition over a space of functions instead of decisions  imposing significant computational implications; see  \citet{Papadimitriou:1987:CMD:2875343.2875347}. 
Finally, in the common information approach \citep{Nayyar2011,Nayyar2013b}, the subsystems share a subset of their past observations and decisions to a shared memory accessible by all subsystems.  The solution is derived by reformulating the problem from the viewpoint of a ``coordinator" with access only to the shared information (the common information), whose task is to provide ``prescription" strategies to each subsystem. The coordinator's problem is a centralized stochastic control problem.

\subsubsection{Learning-based Control} 

Adaptive control methods \citep{Narendra1989StableAS,Sastry1989AdaptiveCS,strm1989AdaptiveC,Ioannou2012RobustAC} have successfully addressed regulation and tracking control problems with safety guarantees 
by accommodating model uncertainties; see \citet{Dydek2013AdaptiveCO,Leman2009L1AC}.
Reinforcement learning (RL) has  emerged from machine learning as an adaptive approach to control dynamical systems; \citet{Bertsekas1996,Sutton1998a}.
Several efforts have focused on safe learning approaches combining robust reachability guarantees from control theory with Bayesian analysis based on empirical observations \citep{Fisac:2019wb}, and on  learning the system's unknown dynamics based on a Gaussian process model to iteratively approximate the maximal safe set; see \citet{Akametalu:2014th}. 
Iterative learning control \citep{Armstrong:2021vw}, has been also widely used for system identification, or in conjunction with extremum seeking \citep{Khong:2016ug,Khong:2016ws}, for recursively constructing an input such that the corresponding system output tracks a prescribed reference trajectory closely. In communication networks, where models of wireless channels  are available only through data samples \citep{Gatsis:2021tt} there have been efforts on learning approximately optimal power allocation policies to maximize control performance of a set of independent control systems within a fixed budget; see \citet{Eisen:2018ws}.

Other research efforts over the years have focused on developing robust learning-based approaches in applications related to quadrotor safety and steady-state stability \citep{Aswani:2013ue}, learning-based model predictive control \citep{Rosolia:2018wv},  real-time learning  \citep{Malikopoulos2009b} of powertrain operation of vehicles with respect to the driver's driving style \citep{Malikopoulos2010a}, learning for traffic control in simulation \citep{Wu:2017uz} in conjunction with transfer of learned policies from simulation to a scaled environment \citep{chalaki2020ICCA}, decentralized learning for stochastic games \citep{Arslan:2017vo}, learning for optimal social routing  \citep{Krichene:2018we} and congestion games \citep{Krichene:2015vx}, and  learning for enhanced security against replay attacks in CPS; see \citet{Zhai:2021wy,Sahoo:2020tx}. 

Regularities of optimal control on the space of transition kernels along with the implications on robustness of optimal control strategies derived using an ``incorrect" model and applied to the actual system have been discussed by \citet{Kara:2018vu}. Approximate planning and learning in partially observed systems using an information state was more recently proposed by \citet{Subramanian2020ApproximateIS}. Alternatively, one can  establish  an approximate information state, defined in terms of properties that can be estimated using sampled trajectories, along with an approximate dynamic program; see \citet{Subramanian2019ApproximateIS}. This approach provides a constructive way for RL in partially observed systems. Other efforts have also combined model reference adaptive control with RL to generate online policies; see \citet{Guha2021OnlinePF}. 
Two recent survey papers by \citet{Kiumarsi:2018tq} and \citet{Recht2018ATO} provide a comprehensive review of the general RL problem formulations along with a complete list of applications.

\subsection{Contributions of This Paper} \label{sec:1c}
In this paper, we consider CPS consisting of several subsystems with a common objective and a nonclassical information structure, where the state of the system is not fully observed. 
We provide a theoretical framework,  which can combine offline model-based control with online learning approaches, to yield the optimal control strategy of the system. 
More specifically, we identify a sufficient information state for the system which does not depend on the control strategy. An important consequence of this independence is that for any given choice of a control strategy and a realization of the system’s variables until time $t,$ the information states at future times do not depend on the choice of the control strategy at time $t$ but only on the realization of the decision at time $t,$ and thus they are related to the concept of separation between estimation of the state and control. Namely, the future information states are separated from the choice of the current control strategy. The adjective ``separated" is used to emphasize the fact that in implementing such an optimal policy, we first need to learn the information state and then choose the control. Such control strategies are called separated control strategies. Hence, we can derive offline the optimal control strategy of the system with respect to the information state, which might not be precisely known due to model uncertainties or complexity of the system, and then use standard learning approaches  to learn the information state online while data are added gradually to the system in real time. 

The contributions of this paper are: (1) the institution of  an information state of the system, which does not depend on the control strategy (Theorem \ref{theo:y_t}), that allows us to restrict attention to separated control strategies; (2) a dynamic programming decomposition that uses a CPS model and the information state to derive offline  optimal  separated control strategies (Theorem \ref{theo:dp}) which are optimal  for the actual system (Theorem \ref{theo:CPSmodel}); and (3)  providing structural properties of the dynamic programming decomposition (Theorem \ref{theo:concave}) which allow us to derive the optimal strategies offline using standard techniques for centralized partially observed Markov decision processes.

The two  features which sharply distinguish the framework presented here from previous learning-based, or combined learning and control approaches reported in the literature to date are the following.
First, the CPS imposes a nonclassical information structure while the state of the system is not fully observed. To the best of our knowledge, this is the first time that results on such systems are derived by separating the control and the learning tasks of the problem. Second, the large volume of data that is added to the system gradually is compressed to sufficient statistics without loss of optimality (Theorem \ref{theo:dp}) which constitutes the information state of the system. Using this information state, we derive results for optimal control strategies in a time-invariant domain. Thus, the volume of  data which is added gradually to the system does not cause the domain of the control strategies to increase with time. 
The latter is quite important since searching and then implementing control strategies with increasing domain is burdensome.

\subsection{Organization of This Paper} \label{sec:1e}

The remainder of the paper proceeds as follows. In Section II, we provide the modeling framework and the formulation of the optimal control problem for a CPS with nonclassical information structure. In Section III, we present the analysis for deriving separated control strategies. In Section IV, we present a simple example to illustrate the proposed framework. Finally, we provide concluding remarks and discuss potential directions for future research in Section V.

%%%%%%%%%%%%%%%%%%%%%%%%%%%%%%%%%%%%
%%%%%%%%%%%%%%%%%%%%%%%%%%%%%%%%%%%%
%SECTION II: Problem Formulation
%%%%%%%%%%%%%%%%%%%%%%%%%%%%%%%%%%%%
%%%%%%%%%%%%%%%%%%%%%%%%%%%%%%%%%%%%
\section{Problem Formulation} 
\label{sec:2}

\subsection{Notation}
Subscripts denote time, and superscripts index subsystems. We denote random variables with upper case letters, and their realizations with lower case letters, e.g., for a random variable $X_t$, $x_t$ denotes its realization. The shorthand notation $X_{t}^{1:K}$ denotes the vector of random variables $\big(X_{t}^1, X_{t}^2,\ldots,X_{t}^K\big)$, $x_{t}^{1:K}$ denotes the vector of their realization $\big(x_{t}^1, x_{t}^2,\ldots,x_{t}^K\big)$, and $h^{1:K}_t(\cdot,\cdot)$ denotes the vector of functions $\big(h^1_t(\cdot,\cdot),\ldots, h^K_t(\cdot,\cdot)\big)$. The expectation of a random variable is denoted by $\mathbb{E}[\cdot]$, the probability of an event is denoted by $\mathbb{P}(\cdot)$, and the probability density function is denoted by $p(\cdot)$. 
For a control strategy $\bf{g}$, we use $\mathbb{E}^{\bf{g}}[\cdot]$, $\mathbb{P}^{\bf{g}}(\cdot)$, and $p^{\bf{g}}(\cdot)$ to denote that the expectation, probability, and probability density  function, respectively, depend on the choice of the control strategy $\bf{g}$. For two measurable spaces $(\mathcal{X}, \mathscr{X})$ and $(\mathcal{Y}, \mathscr{Y})$, $\mathscr{X}\otimes\mathscr{Y}$ is the product $\sigma$-algebra on $\mathcal{X}\times \mathcal{Y}$ generated  by the collection of all measurable rectangles, i.e., $\mathscr{X}\otimes\mathscr{Y}\coloneqq \sigma(\{A\times B: A\in\mathscr{X}, B\in\mathscr{Y} \})$. The product of  $(\mathcal{X}, \mathscr{X})$ and $(\mathcal{Y}, \mathscr{Y})$ is the measurable space $(\mathcal{X}\times \mathcal{Y}, \mathscr{X}\otimes\mathscr{Y})$. We denote the Cartesian product of the sets $\mathcal{G}^k$, $k \in \mathcal{K}$, $\mathcal{K}=\{1,\ldots,K\}$, $K\in\mathbb{N}$, with $\times_{k\in\mathcal{K}}\mathcal{G}^k$.

\subsection{Proposed Approach}

We  consider a CPS representing a system of subsystems with an informationally decentralized structure in which there is  a large volume of data of a dynamic nature that is added to the system gradually and not altogether in advance. 
For such systems, using model-based control approaches cannot effectively facilitate optimal solutions with performance guarantees due to the discrepancy between the model and the actual CPS. On the other hand, since there is  a large volume of data of a dynamic nature that is added to the system gradually in real time, traditional supervised learning approaches might not facilitate robust solutions using data derived offline. By contrast, applying reinforcement learning approaches directly to the actual CPS might impose significant implications on safety and robust operation of the system. 

\begin{figure}
	\centering
	\includegraphics[width=.9\linewidth, keepaspectratio]{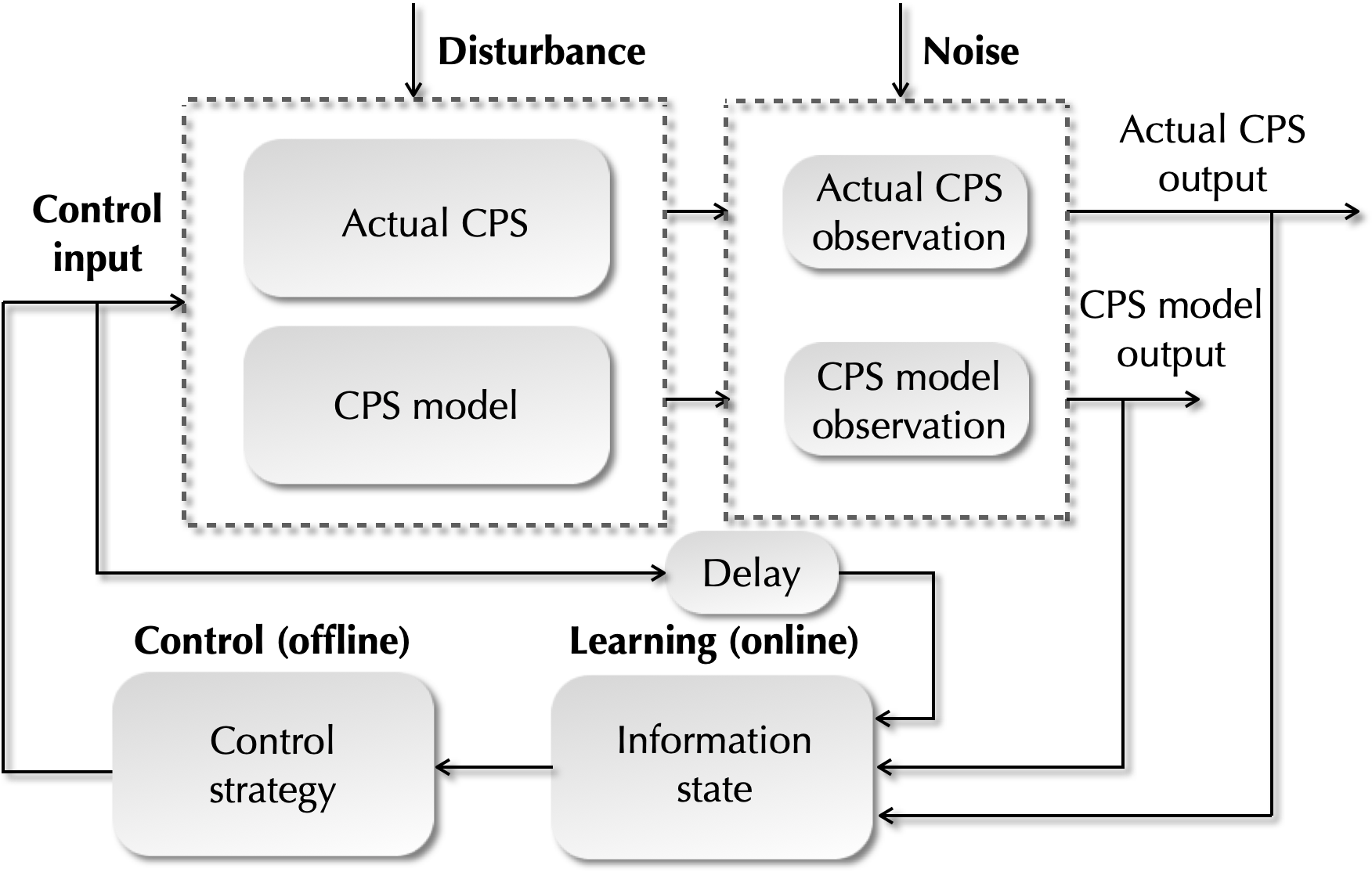} 
	\caption{Illustration of the proposed framework.}%
	\label{fig:1}%
\end{figure}

To address these challenges, our framework aims at separating the control and learning tasks which eventually allows us to combine offline model-based control with online learning approaches.  In particular, we aim at identifying a sufficient information state for the CPS that takes values in a time-invariant space, and  use this information state to derive separated control strategies. Separated control strategies are related to the concept of separation between the estimation of the information  state and control of the system. An important consequence of this separation is that for any given choice of control strategies and a realization of the system's variables until time $t$, the information states of the system at future times do not depend on the choice of the control strategy at time $t$ but only on the realization of the control at time $t$; see \cite{Kumar1986}. Thus, the future information states are separated from the choice of the current control strategy. By establishing separated control strategies, we can derive offline the optimal control strategy of the system with respect to the information state, which might not be precisely known due to model uncertainties or complexity of the system, and then use learning methods  to learn the information state online while data are added gradually to the system in real time.

More specifically, in the proposed framework illustrated in Fig. \ref{fig:1}, we use the actual CPS, i.e., the actual system that we seek to optimally control online, in parallel with a  model of the CPS that is available. The main idea here is the institution of an information state which is the conditional joint probability distribution of the states of the CPS model and the actual CPS at time $t$ given all data available of the model up until time $t$, i.e., $p(\text{state of CPS model, state of actual CPS}~ |$ $~ \text{data of the CPS model})$. 
We use this information state along with the CPS model to derive offline separated control strategies.
Since we derive the optimal strategies offline, the state of the actual CPS  is not known, i.e., the actual CPS operates only online, and thus the optimal strategy of the CPS model is parameterized with respect to all realizations of the  state of the actual CPS. However,  the control strategy and the process of estimating the information state are separated. Therefore, we can learn the information state of the system online, while we operate simultaneously the CPS model and the actual CPS in real time. Namely,  the optimal strategy derived for the CPS model offline, which is parameterized with respect to the state of the actual CPS, is used to operate the actual CPS in parallel with the CPS model. As we collect data from the two systems, we can learn the information state online.
In our exposition, we show that when the information state becomes known online through learning, the separated control strategy of the CPS model derived offline is optimal for the actual CPS (Theorem \ref{theo:CPSmodel}). 
The framework described above is centralized, e.g., a central controller controls all subsystems.

\subsection{Modeling Framework}

We  consider a CPS consisting of $K \in \mathbb{N}$ subsystems with a measurable state space $(\mathcal{X}_t, \mathscr{X}_t)$, where $\mathcal{X}_t$ is the set in which the CPS state takes values at time $t = 0,1,\ldots,T$, $T\in\mathbb{N}$, and $\mathscr{X}_t$ is the associated $\sigma$-algebra. 
Let $X_t$ be a random variable that represents the state of the CPS model and $\hat{X}_t$ be a random variable that represents the state of the actual CPS. Both  random variables are defined on the probability space $(\Omega, \mathscr{F}, \mathbb{P})$, i.e., $X_t: (\Omega, \mathscr{F})\to(\mathcal{X}_t, \mathscr{X}_t)$,  $\hat{X}_t: (\Omega, \mathscr{F})\to(\mathcal{X}_t, \mathscr{X}_t)$, where $\Omega$ is the sample space, $\mathscr{F}$ is the associated $\sigma$-algebra, and $\mathbb{P}$ is a probability measure  on $(\Omega, \mathscr{F})$. The control of each subsystem $k \in \mathcal{K}$, $\mathcal{K}=\{1,\ldots,K\}$, is represented by a random variable $U_t^k: (\Omega, \mathscr{F})\to(\mathcal{U}_t^k, \mathscr{U}_t^k),$ defined on the probability space $(\Omega, \mathscr{F}, \mathbb{P})$, and takes values in the measurable space $(\mathcal{U}^k_t, \mathscr{U}^k_t)$,  where $\mathcal{U}^k_t$ is  subsystem $k$'s nonempty feasible set of actions at time $t$ and $\mathscr{U}^k_t$ is the associated $\sigma$-algebra.
Let ${U}_t^{1:K}=(U_t^1,\ldots,U_t^K)$ be the control of CPS at time $t$.  Starting at the initial state $X_0$, the evolution of the CPS model is described by the state equation
\begin{align}\label{eq:state}
	X_{t+1}=f_t\big(X_t,U_t^{1:K},W_t\big), 
\end{align}
where  $t = 0,1,\ldots,T-1$, and $W_t$ is a random variable defined on the probability space $(\Omega, \mathscr{F}, \mathbb{P})$ that corresponds to the external, uncontrollable disturbance to the CPS and takes values in a measurable set $(\mathcal{W}, \mathscr{W})$, i.e., $W_t:(\Omega, \mathscr{F})\to(\mathcal{W}, \mathscr{W})$. 
Similarly, starting at the initial state $\hat{X}_0$, the evolution of the actual CPS  is described by the state equation
\begin{align}\label{eq:statereal}
	\hat{X}_{t+1}=\hat{f}_t\big(\hat{X}_t,U_t^{1:K},W_t\big),
\end{align}
where $t = 0,1,\ldots,T-1$, while $\{W_t: t=0,\ldots,T-1\}$ is a sequence of independent random variables that are also independent of the initial states $X_0$ and $\hat{X}_0$.
At time $t = 0,1,\ldots,T-1$, every subsystem $k\in\mathcal{K}$ in the model makes an observation $Y_t^k$, which takes values in a measurable set $(\mathcal{Y}^k, \mathscr{Y}^k)$, described by the observation equation 
\begin{align}\label{eq:observe}
	Y_t^k=h_t^k(X_t,Z_t^k),
\end{align}
where $Z_t^k$ is a random variable defined on the probability space $(\Omega, \mathscr{F}, \mathbb{P})$ that corresponds to the noise of each subsystem's sensor and takes values in a measurable set $(\mathcal{Z}^k, \mathscr{Z}^k)$, i.e., $Z_t^k:(\Omega, \mathscr{F})\to(\mathcal{Z}^k, \mathscr{Z}^k)$, while  $\{Z_t^k: ~t=0,\ldots,T-1;~ k=1,\ldots,K\}$ is a sequence of independent random variables that are also independent of $\{W_t: t=0,\ldots,T-1\}$, and the initial states $X_0$ and $\hat{X}_0$. Similarly, at time $t = 0,1,\ldots,T-1$, every subsystem $k\in\mathcal{K}$ in the actual CPS makes an observation $\hat{Y}_t^k$, which takes values in a measurable set $(\mathcal{Y}^k, \mathscr{Y}^k)$, described by the observation equation 
\begin{align}\label{eq:observereal}
	\hat{Y}_t^k=h_t^k(\hat{X}_t,Z_t^k).
\end{align}
We consider that the actual CPS has $n$-step delayed information sharing, i.e., at time $t$, subsystem $k\in\mathcal{K}$ observes  $\hat{Y}_t^k$, and the $n$-step past observations $\hat{Y}_{0:t-n}^{1:K}$ and decisions $U_{0:t-n}^{1:K}$ of the entire system. 
At time $t$, the data available to subsystem $k$ consist of the data $\hat{\Delta}_t$ available to all subsystems, i.e.,
\begin{align}\label{eq:delta}
	\hat{\Delta}_t\coloneqq (\hat{Y}_{0:t-n}^{1:K}, U_{0:t-n}^{1:K}),
\end{align}
where $\hat{Y}_{0:t-n}^{1:K}=\{\hat{Y}_{0:t-n}^{1},\ldots,\hat{Y}_{0:t-n}^{K}\}$, $U_{0:t-n}^{1:K}=\{U_{0:t-n}^{1},$ $\ldots,U_{0:t-n}^{K}\}$, and  the data $\Lambda_t^k$ known only to subsystem $k\in\mathcal{K},$ is given by
\begin{align}\label{eq:lambda}
	\hat{\Lambda}_t^k\coloneqq (\hat{Y}_{t-n+1:t}^{k}, U_{t-n+1:t-1}^{k}).
\end{align}
Note that the $n$-step delayed information sharing can  also be asymmetric, i.e., for each member $k\in\mathcal{K}$, $\hat{Y}_{t-n_k}^{k}$, $U_{t-n_k}^{k},$ where $n_k\in\mathbb{R}$ is constant but not necessarily the same for each  $k$.
The collection $\{(	\hat{\Delta}_t, 	\hat{\Lambda}_t^k);$ $k\in\mathcal{K};~ t=0,\ldots,T-1\}$, is the information structure of the actual CPS and captures which subsystem knows what about the status of the CPS and when.  In what follows, the results hold for any special case of potential information structures that can be:
\begin{enumerate}
\item \textbf{Periodic information sharing with period $\omega\ge1$:}
In this case \cite{Ooi:1997aa}, for $\alpha= 1,2, \ldots$ and $\alpha\omega< t \le (\alpha +1)\omega$, the pair of
$\hat{\Delta}_t,$ 	$\hat{\Lambda}_t^k$, $k\in\mathcal{K},$ becomes
\begin{align}\label{eq:info1}
	&\hat{\Delta}_t\colon= (\hat{Y}_{0:\alpha\omega}^{1:K}, U_{0:\alpha\omega}^{1:K}),\\
	&\hat{\Lambda}_t^k\colon= (\hat{Y}_{\alpha\omega+1:(\alpha +1)\omega}^{k}, U_{\alpha\omega+1:(\alpha +1)\omega}^{k}).
\end{align}
\item \textbf{$n$-step delayed observation sharing:}
In this case \cite{aicardi1987decentralized}, $\hat{\Delta}_t$ and $\hat{\Lambda}_t^k$, $k\in\mathcal{K},$ become
\begin{align}\label{eq:info2}
	&\hat{\Delta}_t\colon= (\hat{Y}_{0:t-n}^{1:K}),\\
	&\hat{\Lambda}_t^k\colon= (\hat{Y}_{t-n+1:t}^{k}, U_{0:t-1}^{k}).
\end{align}
\item  \textbf{$n$-step delayed control sharing:}
In this case \cite{Bismut:1973aa}, $\hat{\Delta}_t$ and $\hat{\Lambda}_t^k$, $k\in\mathcal{K},$ become
\begin{align}\label{eq:info3}
	&\hat{\Delta}_t\colon= (U_{0:t-n}^{1:K}),\\
	&\hat{\Lambda}_t^k\colon= (\hat{Y}_{0:t}^{k}, U_{t-n+1:t-1}^{k}).
\end{align}
\item  \textbf{No sharing information:}
In this case, $\hat{\Delta}_t$ and $\hat{\Lambda}_t^k$, $k\in\mathcal{K},$ become
\begin{align}\label{eq:info4}
	&\hat{\Delta}_t\colon= \emptyset,\\
	&\hat{\Lambda}_t^k\colon= (\hat{Y}_{0:t}^{k}, U_{0:t-1}^{k}).
\end{align}
\end{enumerate}
The CPS model imposes the same information structure as the actual CPS. The collection $\{(\Delta_t, \Lambda_t^k);$  $k\in\mathcal{K};~ t=0,\ldots,T-1\}$, is the information structure of the model.

\subsection{Optimal Control Problem}
%%%%%%%%%%%%%%%%%%%%%%%%%%%%%%%%%
%%%%%%%%Optimal Control Problem
%%%%%%%%%%%%%%%%%%%%%%%%%%%%%%%%%
Let $(\mathcal{D}_t, \mathscr{D}_t)$  be the measurable spaces of all possible realizations of $\Delta_t$ and $\hat{\Delta}_t$, and $(\mathcal{L}_t^k, \mathscr{L}_t^k), k\in\mathcal{K},$  be the measurable spaces of all possible realizations of $\Lambda_t^k$ and $\hat{\Lambda}_t^k$, where $\mathscr{D}_t$ and $\mathscr{L}_t^k$ are the associated $\sigma$-algebras. 
A control strategy $\textbf{g}=\{g_t;~ t=0,\ldots,T-1\}$, $\textbf{g}\in\mathcal{G}$,  $\mathcal{G}= : (\mathcal{L}_t^1\times\dots\times\mathcal{L}_t^K\times \mathcal{D}_t, \mathscr{L}_t^1\otimes\dots\otimes\mathscr{L}_t^K\otimes\mathscr{D}_t)$ yields a decision 
\begin{align}\label{eq:control}
	U^{1:K}_t =g_t(\hat{\Delta}_t, \hat{\Lambda}^{1:K}_{t}),
\end{align}
where  the measurable function $g_t$ is the control law.
\begin{problem}[Actual CPS] \label{problem1}
	The problem is to derive the optimal control strategy $\textbf{g}^*\in\mathcal{G}$ that minimizes the expected total cost of the actual CPS, 
	\begin{align}\label{eq:cost}
		\hat{J}(\textbf{g})=\mathbb{E}^{\textbf{g}}\left[\sum_{t=0}^{T-1} c_t(\hat{X}_t, U_t^{1:K})+c_T(\hat{X}_T)\right],
	\end{align}
	where the expectation is with respect to the joint probability distribution of the random variables $\hat{X}_t$ and  $U_t^{1:K}$ designated by the choice of $\textbf{g}\in\mathcal{G}$, $c_t(\cdot, \cdot):(\mathcal{X}_t\times \prod_{k\in\mathcal{K}} \mathcal{U}_t^k, \mathscr{X}_t \otimes \mathscr{U}_t^1\otimes\cdots\otimes\mathscr{U}_t^K)\to\mathbb{R}$ 
	is the measurable cost function of the actual CPS at $t$, and $c_T(\cdot):(\mathcal{X}_T, \mathscr{X}_T) \to\mathbb{R}$ is the measurable cost function at $T$. 
\end{problem}

The statistics of the primitive random variables $\hat{X}_0$,  $\{W_t: t=0,\ldots,T-1\}$, $\{Z_t^k: k\in\mathcal{K};~ t=0,\ldots,T-1\}$, the observation equations $\{h^k_t: k\in\mathcal{K};~ t=0,\ldots,T-1\}$, and the cost functions $\{c_t: t=0,\ldots,T\}$ are all known. However, the state equations $\{\hat{f}_t: t=0,\ldots,T-1\}$ are not known.

%%%%%%%%%%%%%%%%%%%%%%%%%%%%%%%%%%%%%%%%%%%%%%%%%%%%%%
%%%%%%%%%%%%%%%%%%%%%%%%%%%%%%%%%%%%%%%%%%%%%%%%%%%%%%
%SECTION III: Separation of Learning and Control from a Centralized Perspective
%%%%%%%%%%%%%%%%%%%%%%%%%%%%%%%%%%%%%%%%%%%%%%%%%%%%%%
%%%%%%%%%%%%%%%%%%%%%%%%%%%%%%%%%%%%%%%%%%%%%%%%%%%%%%

\section{Separation of Learning and Control } \label{sec:3}
In our exposition, we address Problem \ref{problem1} from the point of view of a central controller who seeks to derive the optimal strategy $\textbf{g}\in\mathcal{G}$ of the actual CPS.
First, we institute an appropriate information state, defined formally next, that can be used to formulate a classical dynamic programming decomposition. To establish this information state, we use the CPS model in conjunction with the actual CPS (Fig. \ref{fig:2}).

\begin{figure}
	\centering
	\includegraphics[width=.9\linewidth, keepaspectratio]{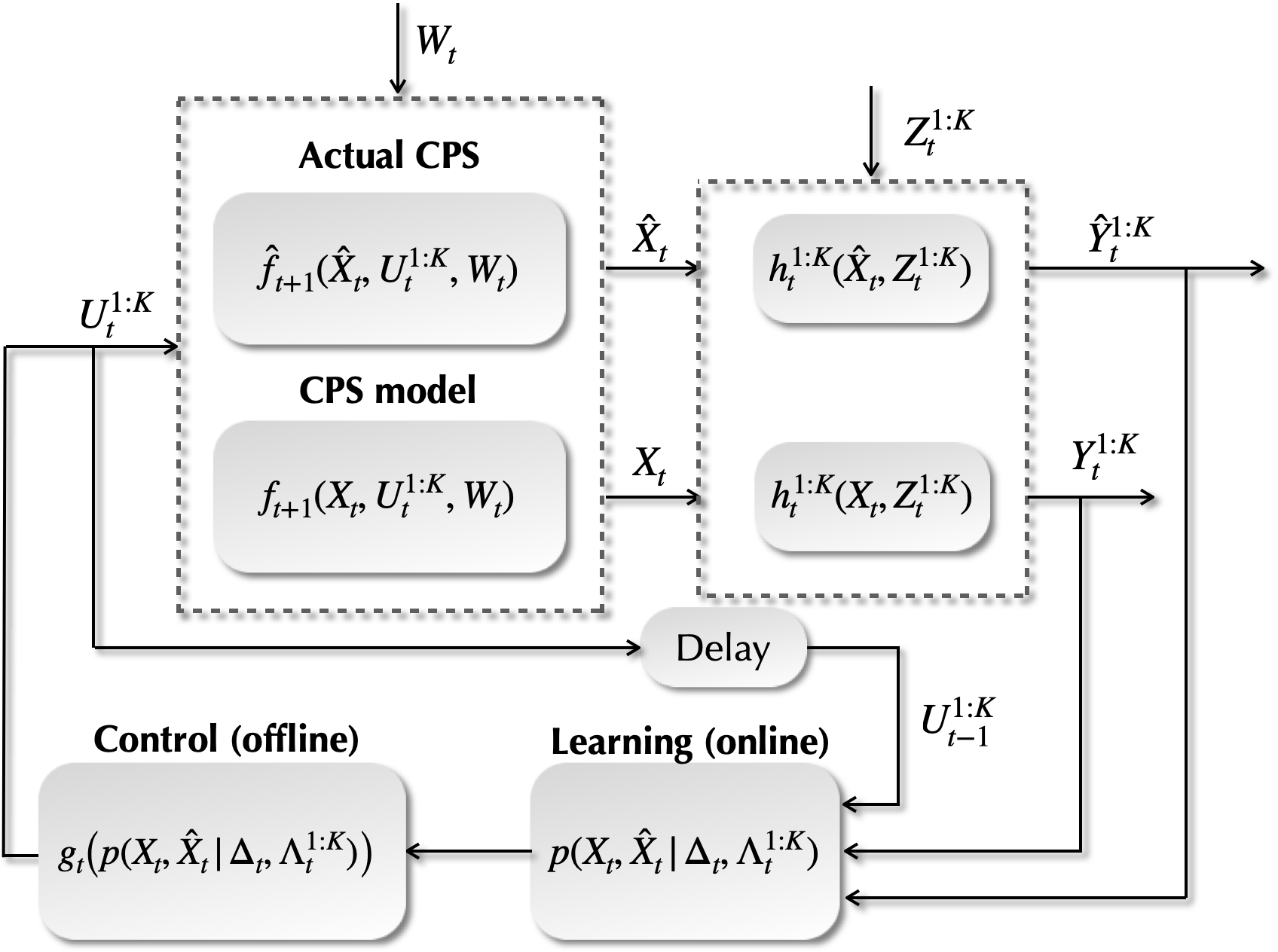} 
	\caption{Separation of learning and control.}%
	\label{fig:2}%
\end{figure}

%%%%%%%%%%%%
%%%%Definition
%%%%%%%%%%%%
\begin{definition} \label{def:infoteam}
	An information state, $\Pi_t$, for the system described by the state equations \eqref{eq:state} and \eqref{eq:statereal}, is (a) a function of  $(\Delta_t, \Lambda_t^{1:K})$, while (b) $\Pi_{t+1}$ is determined from $\Pi_t$, $Y_{t+1}^{1:K}$, and $U_{t}^{1:K}$.
\end{definition}

We consider densities for all probability distributions to simplify notation. Let $\textbf{g}=\{g_t;~ t=0,\ldots,T-1\}$, $\textbf{g}\in\mathcal{G},$ be a control strategy and $(\Delta_t, \Lambda_t^{1:K})$ be the information structure of the CPS model. The control strategy $\textbf{g}$ yields a decision $U^{1:K}_t =g_t(\Delta_t, \Lambda^{1:K}_{t})$.

Before we proceed with establishing the information state, we prove some essential properties. 

%%%%%%%%%%%%
%%%%Lemma
%%%%%%%%%%%%
\begin{lemma} \label{lem:y_t}
	For any control strategy $\textbf{g}\in\mathcal{G}$ of the system, 
	\begin{align}\label{eq:y_t}
		p^{\textbf{g}}(Y^{1:K}_{t+1}~|~X_{t+1}, \hat{X}_{t+1}, \Delta_{t}, \Lambda^{1:K}_t, U^{1:K}_t)= p(Y^{1:K}_{t+1}~|~X_{t+1}),
	\end{align}
	for all $t=0,1,\ldots, T-1.$
\end{lemma}
\begin{proof}
The realization of  $Y^{1:K}_{t+1}$ is statistically determined by the conditional distribution of $Y^{1:K}_{t+1}$ given $X_{t+1}$ in \eqref{eq:observe}, hence
\begin{align}\label{eq:lem1a}
	p^{\textbf{g}}(Y^{1:K}_{t+1}~|~X_{t+1},\hat{X}_{t+1}, \Delta_{t}, \Lambda^{1:K}_t, U^{1:K}_t)= p^{\textbf{g}}(Y^{1:K}_{t+1}~|~X_{t+1}).
\end{align}		
However,
\begin{align}\label{eq:lem1b}
	p^{\textbf{g}}(Y^{1:K}_{t+1}~|~X_{t+1}) = p^{\textbf{g}}(Z^{1:K}_{t+1}\in \prod_{k\in\mathcal{K}} B^k~|~X_{t+1}),
\end{align}	
where $B^k\in \mathscr{Z}^k$, $k\in\mathcal{K}$. Since, $\{Z_{t}^k:~k=1,\ldots,K;~ t=0,\ldots,T-1\}$ is a sequence of independent random variables that are independent of $X_{t+1}$, 
\begin{align}\label{eq:lem1c}
	p^{\textbf{g}}(Z^{1:K}_{t+1}\in \prod_{k\in\mathcal{K}} B^k~|~X_{t+1}) = p(Z^{1:K}_{t+1}\in \prod_{k\in\mathcal{K}} B^k).
\end{align}
Hence, 
\begin{align}\label{eq:lem1d}
	p^{\textbf{g}}(Y^{1:K}_{t+1}~|~X_{t+1}) = p(Y^{1:K}_{t+1}~|~X_{t+1}).
\end{align} 
The result follows from \eqref{eq:lem1a} and \eqref{eq:lem1d}.
\end{proof}

%%%%%%%%%%%%
%%%%Lemma
%%%%%%%%%%%%
\begin{lemma} \label{lem:x_t1ut}
	For any control strategy $\textbf{g}\in\mathcal{G}$ of the system, 
	\begin{align}
		&p^{\textbf{g}}(X_{t+1},\hat{X}_{t+1}~|~X_t, \hat{X}_{t},  \Delta_{t}, \Lambda^{1:K}_t, U^{1:K}_t) \nonumber\\
		&= p(X_{t+1}, \hat{X}_{t+1}~|~X_t, \hat{X}_{t}, U^{1:K}_t), \label{eq:x_t1ut}
	\end{align}
	for all $t=0,1,\ldots, T-1.$
\end{lemma}
\begin{proof}
	The realization of  $X_{t+1}$ is statistically determined by the conditional distribution of $X_{t+1}$ given $X_{t}$ and $U^{1:K}_{t}$, i.e., $p^{\textbf{g}}(X_{t+1}~|~X_t,  U^{1:K}_t) $. 
	Similarly, the realization of $\hat{X}_{t+1}$ is statistically determined by the conditional distribution of $\hat{X}_{t+1}$ given $\hat{X}_{t}$ and $U^{1:K}_{t}$, i.e., $p^{\textbf{g}}(\hat{X}_{t+1}~|~\hat{X}_t,  U^{1:K}_t) $. 

	From \eqref{eq:state}, we have
	\begin{gather}
		p^{\textbf{g}}(X_{t+1}~|~X_t,  U^{1:K}_t) = p^{\textbf{g}}(W_{t}\in A~|~X_t,  U^{1:K}_t), \label{eq:lem2a}
	\end{gather}		
	where $A\in\mathscr{W}$. From \eqref{eq:statereal}, we have
	\begin{gather}
		p^{\textbf{g}}(\hat{X}_{t+1}~|~\hat{X}_t,  U^{1:K}_t) = p^{\textbf{g}}(W_{t}\in A~|~\hat{X}_t,  U^{1:K}_t), \label{eq:lem2areal}
	\end{gather}		
	where $A\in\mathscr{W}$. 	
	Since, $\{W_{t}: t=0,\ldots,T-1\}$ is a sequence of independent random variables that are independent of $X_{t}$, $\hat{X}_t$, and $U^{1:K}_{t}$, 
	\begin{align}
		p^{\textbf{g}}(W_{t}\in A~|~X_t,  U^{1:K}_t) &= p^{\textbf{g}}(W_{t}\in A~|~\hat{X}_t,  U^{1:K}_t) \nonumber\\ &= p(W_{t}\in A).\label{eq:lem2b}
	\end{align}
	
	Next,
	\begin{align}
		&p^{\textbf{g}}(X_{t+1}~|~X_t, \Delta_{t}, \Lambda^{1:K}_t, U^{1:K}_t) \nonumber\\
		&= p^{\textbf{g}}(W_{t}\in A~|~X_t, \Delta_{t}, \Lambda^{1:K}_t, U^{1:K}_t) = p(W_{t}\in A). \label{eq:lem2c}
	\end{align}	
	Similarly,
	\begin{align}
		&p^{\textbf{g}}(\hat{X}_{t+1}~|~\hat{X}_t, \Delta_{t}, \Lambda^{1:K}_t, U^{1:K}_t) \nonumber\\
		&= p^{\textbf{g}}(W_{t}\in A~|~\hat{X}_t, \Delta_{t}, \Lambda^{1:K}_t, U^{1:K}_t) = p(W_{t}\in A). \label{eq:lem2creal}
	\end{align}	
	
	The result follows from \eqref{eq:lem2a}, \eqref{eq:lem2areal}, \eqref{eq:lem2b}, \eqref{eq:lem2c}, and \eqref{eq:lem2creal}.	
\end{proof}

%%%%%%%%%%%%
%%%%Lemma
%%%%%%%%%%%%
\begin{lemma} \label{lem:x_t}
	For any control strategy $\textbf{g}\in\mathcal{G}$ of the system, 
	\begin{align}\label{eq:x_t}
		p^{\textbf{g}}(X_{t}, \hat{X}_{t}~|~\Delta_{t}, \Lambda^{1:K}_t) = p(X_{t},\hat{X}_{t}~|~\Delta_{t}, \Lambda^{1:K}_t),
	\end{align}
	for all $t=0,1,\ldots, T-1.$
\end{lemma}
\begin{proof} 
	By expanding $p^{\textbf{g}}(X_{t},\hat{X}_{t}~|~\Delta_{t}, \Lambda^{1:K}_t)$, we have
	\begin{align}
		&p^{\textbf{g}}(X_{t},\hat{X}_{t}~|~\Delta_{t}, \Lambda^{1:K}_t)\nonumber\\
		&= p^{\textbf{g}}(X_{t},\hat{X}_{t}~|~\Delta_{t-1}, \Lambda^{1:K}_{t-2}, Y^{1:K}_{t-1}, Y^{1:K}_t, U^{1:K}_{t-2},U^{1:K}_{t-1}). \label{eq:lem4a}
	\end{align}		
	However, the realizations of $X_{t}$ and $\hat{X}_{t}$ are statistically determined by the conditional joint distribution of  $X_{t}$ and $\hat{X}_{t}$  given $X_{t-1}$, $\hat{X}_{t-1}$ and $U^{1:K}_{t-1}$, which does not depend on the control strategy $\textbf{g}$ (Lemma \ref{lem:x_t1ut}), so we can drop the superscript in \eqref{eq:lem4a}, and thus \eqref{eq:x_t} follows immediately.	
\end{proof}

%%%%%%%%%%%%
%%%%Remark
%%%%%%%%%%%%
\begin{remark}\label{cor:lemU}
	As a consequence of Lemma \ref{lem:x_t}, and since both $X_{t}$ and $\hat{X}_{t}$ do not depend on $U^{1:K}_t$, we have
	\begin{align}\label{eq:lemU}
		p^{\textbf{g}}(X_{t}, \hat{X}_t~|~\Delta_{t}, \Lambda^{1:K}_t, U^{1:K}_t) = p(X_{t}, \hat{X}_t~|~\Delta_{t}, \Lambda^{1:K}_t).
	\end{align}
\end{remark}
Given that we can observe the data $(\Delta_t, \Lambda_t^{1:K})$ of the CPS model, we can compress these data to a sufficient statistic which is the probability density function $p(X_{t}, \hat{X}_t ~|~\Delta_{t}, \Lambda^{1:K}_{t})$, called information state and denoted by $\Pi_{t}(\Delta_{t}, \Lambda^{1:K}_{t})(X_{t},\hat{X}_t)$.
The next result shows that such information state does not depend on the control strategy of the CPS model.

%%%%%%%%%%%%
%%%%Theorem
%%%%%%%%%%%%
\begin{theorem}[Information State of the System] \label{theo:y_t}
	For any control strategy $\textbf{g}\in\mathcal{G}$ derived offline for the CPS model, the information state $\Pi_{t}(\Delta_{t}, \Lambda^{1:K}_{t})(X_{t},\hat{X}_t)$ does not depend on the control strategy $\textbf{g}$.
	Moreover, there is a function $\phi_t$, which does not depend on the control strategy $\textbf{g}$, such that
	\begin{align}\label{eq:xt1}
		\Pi_{t+1}(\Delta_{t+1}, \Lambda^{1:K}_{t+1})&(X_{t+1},\hat{X}_{t+1}) \nonumber\\
		= \phi_t\big[ \Pi^k_{t}(\Delta_{t}, \Lambda^{1:K}_{t})&(X_{t},\hat{X}_t), Y^{1:K}_{t+1}, U^{1:K}_t \big],
	\end{align}
	for all $t=0,1,\ldots, T-1.$
\end{theorem}
\begin{proof}
	See Appendix \ref{app:1}.
\end{proof}

The information state $\Pi_{t+1}(\Delta_{t+1}, \Lambda^{1:K}_{t+1})(X_{t+1},$ $ \hat{X}_{t+1})$ of the system is the entire probability density function and not just its value at any particular realization of  $X_{t+1}$ and $ \hat{X}_{t+1}$. This is because to compute $\Pi_{t+1}(\Delta_{t+1}, \Lambda^{1:K}_{t+1})(X_{t+1},\hat{X}_{t+1})$ for any particular realization of  $X_{t+1}$ and $\hat{X}_{t+1}$, we need the probability density functions $p(~\cdot, \cdot  ~|~ \Delta_{t}, \Lambda^{1:K}_t, U^{1:K}_t)$ and $p(~\cdot, \cdot  ~|~ \Delta_{t}, \Lambda^{1:K}_t)$. This implies that the information state takes values in the space of these probability densities, which is an infinite-dimensional space.

In what follows, to simplify notation, the information state $\Pi_{t}(\Delta_{t}, \Lambda^{1:K}_{t})(X_{t},\hat{X}_{t})$ of the system at $t$ is denoted simply by $\Pi_t$. We use its arguments only if it is required in our exposition.

%%%%%%%%%%%%
%%%%Definition
%%%%%%%%%%%%

\begin{definition}\label{def:septeam}
	A control strategy $\textbf{g}=\{g_t;~ t=0,\ldots,T-1\}$, of the system is said to be \textit{separated} if $g_t$ depends on $\Delta_{t}$ and $\Lambda^{1:K}_{t}$ only through the information state, i.e., $U^{1:K}_t  = g_t\big(\Pi_{t}(\Delta_{t}, \Lambda^{1:K}_{t})(X_{t},\hat{X}_{t})\big)$. Let $\mathcal{G}^s\subseteq\mathcal{G}$ denote the set of all separated control strategies.
\end{definition}

To derive the optimal control strategy of the actual CPS in Problem \ref{problem1}, we formulate the following optimization problem.

%%%%%%%%%%%%
%%%%Problem
%%%%%%%%%%%%
\begin{problem}(CPS model) \label{problem2}
	Using the CPS model, we seek to derive offline the optimal control strategy $\textbf{g}^*\in\mathcal{G}^s$ that minimizes the following expected total cost 
	\begin{align}			
		&J(\textbf{g};\hat{x}_{0:T})\nonumber\\
		&= \mathbb{E}^{\textbf{g}}\Bigg[\sum_{t=0}^{T-1}\Big[c_t(X_t, U^{1:K}_t)
		+ \beta \cdot|X_{t+1}- \hat{X}_{t+1}|^2\Big]\nonumber\\
		&+c_T(X_T) \Bigg], \label{eq:costreal}	
	\end{align}	
	where $X_{t+1}=f_t\big(X_t,U_t^{1:K},W_t\big)$, $\hat{X}_{t+1}=\hat{f}_t\big(\hat{X}_t,$ $U_t^{1:K},W_t\big)$, and $\beta$ is a factor to adjust the units and size of the norm accordingly as designated by the cost function $c_t(\cdot, \cdot)$. The norm penalizes any discrepancy between the realizations of the state of the CPS model and the state of the actual CPS. The expectation in \eqref{eq:costreal}	is with respect to the joint probability distribution of the random variables $X_t$,  $U_t^{1:K}$, $\hat{X}_t$, $t=0,1,\ldots, T,$ (designated by the choice of $\textbf{g}\in\mathcal{G}^s$) and $W_t$.
	Since solving \eqref{eq:costreal}	is an offline process, the realizations $\hat{x}_{0:T}$ of the state $\hat{X}_{t}$, $t = 0,\ldots, T,$ of the actual CPS  are not known, and thus $\textbf{g}^*$ is parameterized with respect to $\hat{x}_{0:T}$. The statistics of the primitive random variables $X_0$,  $\{W_t: t=0,\ldots,T-1\}$, $\{Z_t^k: k\in\mathcal{K};~ t=0,\ldots,T-1\}$, the state equations $\{f_t: t=0,\ldots,T-1\}$, the observation equations $\{h^k_t: k\in\mathcal{K};~ t=0,\ldots,T-1\}$, and the cost functions $\{c_t: t=0,\ldots,T\}$ are all known. 
\end{problem}

Next, we use the information state $\Pi_{t}(\Delta_{t}, \Lambda^{1:K}_{t})$ $(X_{t},\hat{X}_{t})$ to  derive offline the optimal separated control strategy in Problem \ref{problem2}. In our exposition, we define recursive functions, and show that a separated control strategy of the CPS model is optimal. In addition, we obtain a classical dynamic programming decomposition. 

%%%%%%%%%%%%
%%%%Theorem
%%%%%%%%%%%%
\begin{theorem} \label{theo:dp}
	Let $V_t\big(\Pi_{t}(\Delta_{t}, \Lambda^{1:K}_{t})(X_{t},\hat{X}_{t})\big)$ be functions defined recursively for all $\textbf{g}\in\mathcal{G}^s$ by
	\begin{align}
		&V_T\big(\Pi_{T}(\Delta_{T}, \Lambda^{1:K}_{T})(X_{T},\hat{X}_{T})\big) \coloneqq \mathbb{E}^{\textbf{g}}\Big[c_T(X_T)~|\nonumber\\
		&\Pi_{T}=\pi_T \Big],\nonumber\\
		&V_t\big(\Pi_{t}(\Delta_{t}, \Lambda^{1:K}_{t})(X_{t},\hat{X}_{t})\big)\coloneqq \inf_{u^{1:K}_t\in\prod_{k\in\mathcal{K}} \mathcal{U}_t^k }\mathbb{E}^{\textbf{g}}\Big[c_t(X_t,\nonumber\\
		&U^{1:K}_t) 
		+ \beta ~|X_{t+1}  - \hat{X}_{t+1}|^2\nonumber\\
		&+V_{t+1}\big(\phi_t\big[ \Pi_{t}(\Delta_{t}, \Lambda^{1:K}_{t})(X_{t},\hat{X}_{t}), Y^{1:K}_{t+1}, U^{1:K}_t\big]\big)~|~\Pi_{t}=\pi_t, \nonumber\\ 
		&		U^{1:K}_t=u^{1:K}_t  \Big], \label{theo2:1b}		
	\end{align}
	where $c_T(X_T)$ is the cost function at $T$;  $\beta$ is a factor to adjust the units and size of the norm as designated by the cost function $c_t(\cdot, \cdot)$; and $\pi_T$, $\pi_t$, $u^{1:K}_t$ are the realizations of $\Pi_{T}$, $\Pi_{t}$, and $U^{1:K}_t$, respectively.
	Then, (a) for any control strategy $\textbf{g}\in\mathcal{G}^s$,
	\begin{align}			
		&V_t\big(\Pi_{t}(\Delta_{t}, \Lambda^{1:K}_{t})(X_{t},\hat{X}_{t})\big)\le J_t(\textbf{g};\hat{x}_{t:T})\nonumber\\ &\coloneqq \mathbb{E}^{\textbf{g}}\Bigg[\sum_{l=t}^{T-1}\Big[c_l(X_l,U^{1:K}_l)+ \beta \cdot|X_{l+1} - \hat{X}_{l+1}|^2\Big]\nonumber\\					
		&+ c_T(X_T) ~|~\Delta_{t}, \Lambda^{1:K}_{t}  \Bigg],
		\label{theo2:1c}	
	\end{align}	
	where $J_t(\textbf{g};\hat{x}_{t:T})$ is the cost-to-go function of the CPS model, parameterized by the realizations of the  state $\hat{X}_{t}$ of the actual CPS, at time $t$ corresponding to the control strategy $\textbf{g}$; and
	(b) $\textbf{g}\in\mathcal{G}^s$ is optimal and 
	\begin{align}			
		V_t\big(\Pi_{t}(\Delta_{t}, \Lambda^{1:K}_{t})(X_{t},\hat{X}_{t})\big)=J_t(\textbf{g};\hat{x}_{t:T}), \label{theo3:1b}
	\end{align}	
	with probability $1$.
\end{theorem}
\begin{proof}
	See Appendix \ref{app:2}.
\end{proof}

The optimal strategy derived by the CPS model, which is parameterized with respect to the potential realizations $\hat{x}_{0:T}$ of the state $\hat{X}_{t}$, $t = 0,\ldots, T,$  of the actual CPS, is used to operate the actual CPS in parallel with the CPS model (Fig. \ref{fig:2}). As we collect data from the two systems, we learn the information state $\Pi_{t}(\Delta_{t}, \Lambda^{1:K}_{t})(X_{t},\hat{X}_{t})$  online.

%%%%%%%%%%%%
%%%%Proposition
%%%%%%%%%%%%
\begin{proposition} \label{theo:CPSstate}
	The information state $\Pi_{t}(\Delta_{t}, \Lambda^{1:K}_{t})(X_{t},$ $\hat{X}_{t})$ of the system is a function of  $p(X_{t} ~|~ \Delta_{t}, \Lambda^{1:K}_{t})$, $p(\hat{X}_{t} ~|~ \hat{\Delta}_{t}, \hat{\Lambda}^{1:K}_{t})$, and $p(\hat{Y}_{{0:t}}^{1:K}~|~U^{1:K}_{{0:t}-1})$.
\end{proposition}
\begin{proof}
	Recall $\Pi_{t}(\Delta_{t}, \Lambda^{1:K}_{t})(X_{t},\hat{X}_t) =p(X_{t},\hat{X}_t$ $~|~\Delta_{t}, \Lambda^{1:K}_{t})$. Next,
	\begin{align}	
		&p(X_{t},\hat{X}_t~|~\Delta_{t}, \Lambda^{1:K}_{t})\nonumber\\
		&=\frac{p(\hat{X}_t~|~ X_{t},\Delta_{t}, \Lambda^{1:K}_{t})\cdot p(X_t, \Delta_{t}, \Lambda^{1:K}_{t})}{p(\Delta_{t}, \Lambda^{1:K}_{t})} \nonumber\\
		&=\frac{p(\hat{X}_t~|~ U^{1:K}_{0:t-1})\cdot p(X_t, \Delta_{t}, \Lambda^{1:K}_{t})}{p(\Delta_{t}, \Lambda^{1:K}_{t})} \nonumber\\
		&= p(\hat{X}_t~|~ U^{1:K}_{0:t-1})\cdot p(X_{t} ~|~ \Delta_{t}, \Lambda^{1:K}_{t}),
		\label{theoCPSstate:1b}	
	\end{align}	
	where, in the second equality, we used the fact that $\hat{X}_t$ does not depend on $X_t$ and $Y_{0:t}^{1:K}$, and in the third equality we applied Bayes' rule. The first term in \eqref{theoCPSstate:1b}	 can be written as
	\begin{align}
		&p(\hat{X}_t~|~ U^{1:K}_{0:t-1}) \nonumber\\
		&=\int_{\mathscr{X}_{t}} p(\hat{X}_t~|~\hat{Y}_{0:t}^{1:K}, U^{1:K}_{0:t-1})\cdot p(\hat{Y}_{0:t}^{1:K}~|~ U^{1:K}_{0:t-1}) d\hat{Y}_{0:t}^{1:K},
			\label{theoCPSstate:1c}	
	\end{align}
	and the result follows.
\end{proof}

%%%%%%%%%%%%
%%%%Remark
%%%%%%%%%%%%
\begin{remark} \label{rem:infostate}
	The conditional probabilities $p(X_{t}~|~\Delta_{t}, \Lambda^{1:K}_{t})$ and $p(\hat{X}_{t}~|~\hat{\Delta}_{t}, \hat{\Lambda}^{1:K}_{t})$ can be computed recursively starting from an initial prior $p(X_{0}~|~\Delta_{0}, \Lambda^{1:K}_{0})$ and $p(\hat{X}_{0}~|~\hat{\Delta}_{0}, \hat{\Lambda}^{1:K}_{0})$,
	\begin{align}
		& p(X_{t}~|~\Delta_{t}, \Lambda^{1:K}_{t})\nonumber\\
		& = \theta_{t-1}\big[ p(X_{t-1}~|~\Delta_{t-1}, \Lambda^{1:K}_{t-1}), Y^{1:K}_{t}, U^{1:K}_{t-1} \big],\\
		& p(\hat{X}_{t}~|~\hat{\Delta}_{t}, \hat{\Lambda}^{1:K}_{t})\nonumber\\
		& = \hat{\theta}_{t-1}\big[ p(\hat{X}_{t-1}~|~\hat{\Delta}_{t-1}, \hat{\Lambda}^{1:K}_{t-1}), Y^{1:K}_{t}, U^{1:K}_{t-1} \big],
	\end{align}
	for all $t=0,1,\ldots, T-1,$ where $\theta_{t}$ and $\hat{\theta}_{t}$ are appropriate functions; see \citet{Malikopoulos2021}.
\end{remark}

%%%%%%%%%%%%
%%%%Remark
%%%%%%%%%%%%
\begin{remark} \label{rem:CPSstate}
	The information state $\Pi_{t}(\Delta_{t}, \Lambda^{1:K}_{t})$ $(X_{t},\hat{X}_{t})$ of the system can be obtained by using standard learning approaches, i.e., \citet{Brand:1999aa,Gyorfi:2007aa}, to learn online the conditional probabilities $p(\hat{Y}_{{0:t}}^{1:K}~|~U^{1:K}_{{0:t}-1})$ while we operate the actual CPS. 
\end{remark}

Next, we show that after the information state becomes known through learning, then the separated control strategy of the CPS model derived offline is optimal for the actual CPS.

%%%%%%%%%%%%
%%%%Theorem
%%%%%%%%%%%%
\begin{theorem} \label{theo:CPSmodel}
	Let $\textbf{g}\in\mathcal{G}^s$ be an optimal separated control strategy derived offline for the CPS model which minimizes the expected total cost,
	\begin{align}
		J(\textbf{g};\hat{x}_{0:T})&\coloneqq \mathbb{E}^{\textbf{g}}\Bigg[\sum_{t=0}^{T-1}\Big[ c_t(X_t,U^{1:K}_t)+ \beta\cdot|X_{t+1} \nonumber\\ &- \hat{X}_{t+1}|^2\Big] + c_T(X_T)  \Bigg],
		\label{theo:CPSmodelgo}	
	\end{align}	
	in Problem \ref{problem2}. If $p(X_{t},\hat{X}_{t}~|~\Delta_t,\Lambda_t^{1:K})=\Pi(\Delta_t,\Lambda_t^{1:K})$ $(X_{t+1},\hat{X}_{t+1})$ is known, then $\textbf{g}$ minimizes also the expected total cost of the actual CPS, 
	\begin{align}\label{eq:CPScost}
		\hat{J}(\textbf{g})=\mathbb{E}^{\textbf{g}}\left[\sum_{t=0}^{T-1} c_t(\hat{X}_t, U_t^{1:K})+c_T(\hat{X}_T)\right],
	\end{align}
	in Problem 1.
\end{theorem}
\begin{proof}
	If $p(X_{t},\hat{X}_{t}~|~\Delta_t,\Lambda_t^{1:K})=\Pi(\Delta_t,\Lambda_t^{1:K})(X_{t+1},$ $\hat{X}_{t+1})$ is known, then, for all $t=0,\ldots,T-1$, $U_{t}^{1:K} = g_t\big(\Pi(\Delta_t, \Lambda_t^{1:K})(X_{t+1},\hat{X}_{t+1})\big)$ minimizes \eqref{theo:CPSmodelgo}, which implies 
	\begin{align}			
		|X_{t+1} - \hat{X}_{t+1}|^2 = 0,
		\label{Xhat}	
	\end{align}	
	for all $t=0,\ldots,T-1$, hence $c_t(X_t,U^{1:K}_l) = c_t(\hat{X}_t,U^{1:K}_l)$ and $c_T(X_T) = c_T(\hat{X}_T)$. 
	Therefore,
	\begin{align}			
		J(\textbf{g};\hat{x}_{0:T}) &= \mathbb{E}^{\textbf{g}}\Bigg[\sum_{t=0}^{T-1} c_t(X_t,U^{1:K}_t)
		+c_T(X_T)  \Bigg]\nonumber\\
		&=\mathbb{E}^{\textbf{g}}\Bigg[\sum_{t=0}^{T-1} c_t(\hat{X}_t,U^{1:K}_t)
		+c_T(\hat{X}_T) \Bigg]=\hat{J}(\textbf{g}).
		\label{theoCPS:1c}	
	\end{align}		
\end{proof}

The following results provide some structural properties of the recursive functions.

%%%%%%%%%%%%
%%%%Lemma
%%%%%%%%%%%%
\begin{lemma} \label{lemma:homo}
	The function $V_t\big(\Pi_{t}(\Delta_{t}, \Lambda^{1:K}_{t})(X_{t+1},\hat{X}_{t+1})\big)$ defined recursively in Theorem \ref{theo:dp} is positive homogeneous for all $t=0,\ldots,T$, i.e., for any  $\rho >0$, $V_t\big(\rho~ \Pi_{t}(\Delta_{t}, \Lambda^{1:K}_{t})(X_{t+1},\hat{X}_{t+1})\big)= \rho ~V_t\big(\Pi_{t}(\Delta_{t}, \Lambda^{1:K}_{t})$ $(X_{t+1},\hat{X}_{t+1})\big)$. 
\end{lemma}
\begin{proof}
	See Appendix \ref{app:homo}.
\end{proof}

%%%%%%%%%%%%
%%%%Theorem
%%%%%%%%%%%%
\begin{theorem} \label{theo:concave}
	The function $V_t\big(\Pi_{t}(\Delta_{t}, \Lambda^{1:K}_{t})(X_{t+1},\hat{X}_{t+1})\big)$ defined recursively in Theorem \ref{theo:dp} is concave with respect to $\Pi_{t}(\Delta_{t}, \Lambda^{1:K}_{t})(X_{t+1},\hat{X}_{t+1})$.
\end{theorem}
\begin{proof}
	See Appendix \ref{app:concave}.	
\end{proof}

%%%%%%%%%%%%
%%%%Remark
%%%%%%%%%%%%
\begin{remark}
	From Theorem \ref{theo:concave}, the solution of Problem \ref{problem2} can be derived using standard techniques for centralized partially observed Markov decision processes. If the observation space of the CPS is finite, then \eqref{eq:costreal} has a finite dimensional characterization (see \citet{Krishnamurthy2016}, p. $154$). In particular, the explicit solution to \eqref{eq:costreal} is a piecewise linear concave function of the information state; see \citet{Sondik1971}.
\end{remark}

%%%%%%%%%%%%%%%%%%%%%%%%%%%%%%%%%%%%
%%%%%%%%%%%%%%%%%%%%%%%%%%%%%%%%%%%%
%SECTION IV: Example
%%%%%%%%%%%%%%%%%%%%%%%%%%%%%%%%%%%%
%%%%%%%%%%%%%%%%%%%%%%%%%%%%%%%%%%%%
\section{Illustrative Example} \label{sec:5}

We present a simple example of a system  consisting of two subsystems ($K=2)$ with delayed sharing pattern to illustrate the proposed framework. 
The system evolves for a time horizon $T=4$ while there is a delay $n=2$ on  information sharing between the two subsystems. 
The state of the actual system $\hat{X}_t=(\hat{X}_t^1, \hat{X}_t^2), ~t=1,2,3, 4,$ is two-dimensional, and the initial state (primitive random variable), $\hat{X}_0=(\hat{X}_0^1, \hat{X}_0^2),$ of the system is a Gaussian random variable with zero mean, variance $1$, and covariance $0.5$. 

The state of the actual system evolves as follows:
\begin{align}\label{eq:example1}
	\hat{X}_0 &= (\hat{X}_0^1, \hat{X}_0^2),\\
	\hat{X}_1 &= (\hat{X}_1^1, \hat{X}_1^2)= (\hat{X}_0^1, \hat{X}_0^2),\\
	\hat{X}_2 &= (\hat{X}_2^1, \hat{X}_2^2) = (\hat{X}_0^1 + \hat{X}_0^2, 0),\\
	\hat{X}_3 &= (\hat{X}_3^1, \hat{X}_3^2) = (\hat{X}_2^1, U_2^2)=(\hat{X}_0^1 + \hat{X}_0^2, U_2^2),\\			
	\hat{X}_4 &= (\hat{X}_4^1, \hat{X}_4^2) = (\hat{X}_3^1- \hat{X}_3^2- U_3^1, 0)\nonumber \\
	&= (\hat{X}_0^1 + \hat{X}_0^2- U_2^2-U_3^1, 0),		\label{eq:example1b}
\end{align}
and the observation equations are
\begin{align}\label{eq:example2}
	\hat{Y}_t^k = \hat{X}_{t}^k, \quad k=1,2;~ t=1,2,3, 4.			
\end{align}

The state of the system's model evolves as follows:
\begin{align}\label{eq:example11}
	X_0 &= (X_0^1, X_0^2),\\
	X_1 &= (X_1^1, X_1^2)= (X_0^1, X_0^2),\\
	X_2 &= (X_2^1, X_2^2) = (X_0^1 + X_0^2, 0),\\
	X_3 &= (X_3^1, X_3^2) = (X_2^1, U_2^2) = (X_0^1 + X_0^2, U_2^2),\\			
	X_4 &= (X_4^1, X_4^2) = (X_3^1- X_3^2- U_3^1, 0)\nonumber \\
	&= (X_0^1 + X_0^2- U_2^2-U_3^1, 0),			\label{eq:example11b}
\end{align}
and the observation equations are
\begin{align}\label{eq:example2a}
	Y_t^k = X_{t}^k, \quad k=1,2;~ t=1,2,3,4.			
\end{align}

Since $X_0$ and $\hat{X}_0$ are different, we have implicitly imposed an artificial discrepancy between the model and the actual system.

Each subsystem's feasible sets of actions $\mathcal{U}^k_t$ are specified by
\begin{align}\label{eq:example3}
	\mathcal{U}^k_t =
	\begin{cases}
		\mathbb{R}, \quad  \text{if } {(k,t)} = {{(1,3)\text{ or } (2,2)}}, \\
		{0}, \quad  \text{otherwise}. 
	\end{cases}			
\end{align}
Hence a control strategy $\textbf{g}\in\mathcal{G}^s$ of the system consists only of the pair $\textbf{g}=\{g_2, g_3\}$ since $g_t\equiv 0$ for the remaining $t$. Given the modeling framework above, the information structure $\{(\Delta_t, \Lambda_t^k);$ $k=1,2; ~t=1,2,3\}$ of the system is captured through the model as follows
\begin{align}\label{eq:example4}
	\Delta_1 &=\emptyset, ~\Delta_2 =\emptyset,\\
	\Delta_3 &=\{Y_0^1,Y_0^2, Y_1^1, Y_1^2\} = \{X_0^1, X_0^2\}.\label{eq:example4b}
\end{align}
Note that since $g_1\equiv 0$, the realizations of $U_1^1$ and $U_1^2$ are zero, and thus $\Delta_3$ includes only the observations in \eqref{eq:example4b}. The data $\Lambda_t^k, k=1,2,$ available to subsystem $k$ for the feasible control laws are
\begin{align}\label{eq:example5}
	\Lambda_2^1 &=\{Y_0^1, Y_1^1, Y_2^1\} =\{X_0^1, X_1^1, X_2^1\}\nonumber\\
	&=\{X_0^1,X_0^1+X_0^2\},\\	
	\Lambda_2^2 &=\{Y_0^2, Y_1^2, Y_2^2\} =\{X_0^2, X_1^2, X_2^2\}=\{X_0^2\},\\
	\Lambda_3^1 &=\{Y_2^1, Y_3^1\} =\{X_0^1+ X_0^2, X_0^1+ X_0^2\}\nonumber\\
	&=\{X_0^1+X_0^2\},\\
	\Lambda_3^2 &=\{Y_2^2, Y_3^2,U_2^2\} =\{U_2^2\}\label{eq:example5b}.	
\end{align}

\subsection{Optimal Solution}
The problem is to derive the optimal control strategy $\textbf{g}^*\in\mathcal{G}^s$ of the actual system which is the solution of 
\begin{align}\label{eq:example6}
	J(\textbf{g})&=\min_{u_2^2\in\mathcal{U}_2^2, u_3^1\in\mathcal{U}_3^1}\frac{1}{2}\mathbb{E}^{\textbf{g}}\left[(\hat{X}_3^1)^2 + (U_3^1)^2\right]\nonumber\\
	&=\min_{u_2^2\in\mathcal{U}_2^2, u_3^1\in\mathcal{U}_3^1}\frac{1}{2}\mathbb{E}^{\textbf{g}}\left[(\hat{X}_0^1+\hat{X}_0^2-U_2^2-U_3^1)^2 + (U_3^1)^2 \right].
\end{align}

The feasible set $\mathcal{G}$ of the control strategies of the system consists of all $\textbf{g}=\big\{g_2(\Lambda_2^2, \Delta_2), g_3(\Lambda_3^1, \Delta_3)\big\}$, i.e.,
\begin{align}\label{eq:example7}
	g_2^2&\colon \Delta_2\times \Lambda_2^2 \to U_2^2, ~ \text{ or }~ g_2^2\colon  \hat{X}_0^2 \to \mathbb{R},\\
	g_3^1&\colon \Delta_3\times \Lambda_3^1 \to U_3^1, ~ \text{ or }~ g_3^1\colon  \{\hat{X}_0^1, \hat{X}_0^2\} \to \mathbb{R}.
\end{align}
The problem \eqref{eq:example6} has a unique optimal solution 
\begin{align}\label{eq:example8}
	U_2^2 = \frac{1}{2}\hat{X}_0^2,\quad U_3^1=\frac{1}{2}(\hat{X}_0^1+\hat{X}_0^2) -\frac{1}{4}\hat{X}_0^2.
\end{align}

\subsection{Solution Given by Theorems \ref{theo:dp} and \ref{theo:CPSmodel}}

We solve problem \eqref{eq:example6} by considering the control strategies $\textbf{g}=\{g_t;~t = 0,1,2,3\}$, where the control law is of the form $g_t\big(\Pi(\Delta_t, \Lambda_t^{1:2})$ $(X_t,\hat{X}_t)\big)=g_t\big(\mathbb{P}(X_t,\hat{X}_t~|~\Delta_t, \Lambda_t^{1:2})\big)$. 

For $t=3$, using \eqref{theo2:1b} with $\beta=1$, we have
\begin{align}\label{eq:example9}
	&V_3(\Pi_3) =\min_{u_2^2\in\mathcal{U}_2^2,u_3^1\in\mathcal{U}_3^1}\frac{1}{2}\mathbb{E}^{\textbf{g}}\Big [(X_0^1+X_0^2-U_2^2-U_3^1)^2 \nonumber\\
	&+ (U_3^1)^2 +|X_4-\hat{X}_4|^2~|~\Pi_3\big(\Delta_3, \Lambda_3^{1:2}\big), U_3^{1:2}\Big ]\nonumber\\
	&=\min_{u_2^2\in\mathcal{U}_2^2,u_3^1\in\mathcal{U}_3^1}\frac{1}{2}\mathbb{E}^{\textbf{g}}\Big [(X_0^1+X_0^2-U_2^2-U_3^1)^2 \nonumber\\
	&+ (U_3^1)^2+|X_0^1+X_0^2-\hat{X}_0^1-\hat{X}_0^2|^2~|~\mathbb{P}(X_0^1 + X_0^2, U_2^2, \nonumber\\
	&\hat{X}_0^1 + \hat{X}_0^2~|~X_0^1, X_0^2, X_0^1+X_0^2,U_2^2), U_3^1\Big ],
\end{align}
where, given the information state $\Pi_3$, we can  select the realization of $U_3^1$ that achieves the lower bound in \eqref{eq:example9}.
Hence,
\begin{align}\label{eq:example12}
	U_3^1=\frac{1}{2}(X_0^1+X_0^2) -\frac{1}{2}U_2^2.
\end{align}
Substituting \eqref{eq:example12} into  \eqref{eq:example9} yields
\begin{align}\label{eq:example13}
	&V_3(\Pi_3)= \min_{u_2^2\in\mathcal{U}_2^2,u_3^1\in\mathcal{U}_3^1}\frac{1}{2}\mathbb{E}^{\textbf{g}}\Big [\frac{\big(X_0^1 + X_0^2- U_2^2\big)^2}{2}+|X_0^1\nonumber\\
	&+X_0^2-\hat{X}_0^1-\hat{X}_0^2|^2 ~| ~\mathbb{P}(X_0^1 + X_0^2, U_2^2,\hat{X}_0^1 + \hat{X}_0^2~|~X_0^1, \nonumber\\
	&X_0^2, X_0^1+X_0^2, U_2^2), U_3^1\Big ].
\end{align}

For $t=2$, using \eqref{theo2:1b} with $\beta=1$, we have
\begin{align}\label{eq:example14}
	&V_2(\Pi_2) = \min_{u_2^2\in\mathcal{U}_2^2,u_3^1\in\mathcal{U}_3^1} \frac{1}{2}\mathbb{E}^{\textbf{g}}\Big [V_3(\Pi_3)+|X_3-\hat{X}_3|^2~| \nonumber\\
	&\Pi_2\big(\Delta_2,\Lambda_2^{1:2}\big), U_2^{1:2} \Big ]\nonumber\\
	&= \min_{u_2^2\in\mathcal{U}_2^2,u_3^1\in\mathcal{U}_3^1} \frac{1}{2}\mathbb{E}^{\textbf{g}}\Big [V_3(\Pi_3)+|X_0^1+X_0^2-\hat{X}_0^1-\hat{X}_0^2|^2~|\nonumber\\
	&~\mathbb{P}(X_0^1 + X_0^2, \hat{X}_0^1 + \hat{X}_0^2~|~X_0^1, X_0^1+X_0^2, X_0^2), U_2^2 \Big ]\\
	&= \min_{u_2^2\in\mathcal{U}_2^2,u_3^1\in\mathcal{U}_3^1}\frac{1}{2}\mathbb{E}^{\textbf{g}}\Big [\frac{\big(X_0^1 + X_0^2- U_2^2\big)^2}{2}+2\cdot|X_0^1+X_0^2\nonumber\\
	&-\hat{X}_0^1-\hat{X}_0^2|^2 ~|~\mathbb{P}(X_0^1 + X_0^2, \hat{X}_0^1+\hat{X}_0^2~|~X_0^1, X_0^1+X_0^2,\nonumber\\
	& X_0^2), U_2^2\Big ].
\end{align}
Since
\begin{align}\label{eq:example15}
	&U_2^2 = g_2\big(\mathbb{P}(X_2,\hat{X}_2~|~\Delta_2, \Lambda_2^{1:2})\big) = g_2^2\big(\mathbb{P}(X_0^1 + X_0^2,\nonumber\\
	& \hat{X}_0^1+\hat{X}_0^2~|~X_0^1, X_0^1+X_0^2, X_0^2)\big),
\end{align}
the problem is to choose, for any given $X_0^2$, the estimate of $(X_0^1+X_0^2)$ that minimizes the mean squared error $\big(X_0^1 + X_0^2- U_2^2\big)^2$ in \eqref{eq:example14}. 

Given the Gaussian statistics, the optimal solution is
\begin{align}\label{eq:example16}
	U_2^2  =\frac{1}{2}X_0^2.
\end{align}
Substituting \eqref{eq:example16} into \eqref{eq:example12} yields
\begin{align}\label{eq:example17}
	U_3^1=\frac{1}{2}(X_0^1+X_0^2) -\frac{1}{4}X_0^2.
\end{align}
After learning the information states $\Pi_3\big(\Delta_3, \Lambda_3^{1:2}\big)$ and $\Pi_2\big(\Delta_2, \Lambda_2^{1:2}\big)$, the ``true" values of the initial states in \eqref{eq:example16} and \eqref{eq:example17} corresponding to the actual system become known. Hence, we select $X_0^1=\hat{X}_0^1$ and $X_0^2= \hat{X}_0^2$,  and thus $U_2^2  =\frac{1}{2}\hat{X}_0^2$ and $U_3^1=\frac{1}{2}(\hat{X}_0^1+\hat{X}_0^2) -\frac{1}{4}\hat{X}_0^2.$
Therefore, the control laws of the form $g_t\big(\Pi(\Delta_t, \Lambda_t^{1:K})(X_t,\hat{X}_t)\big)=g_t\big(\mathbb{P}(X_t,\hat{X}_t~|~\Delta_t, \Lambda_t^{1:K})\big)$ yield the unique optimal solution \eqref{eq:example8} of problem \eqref{eq:example6}.

%%%%%%%%%%%%%%%%%%%%%%%%%%%%%%%%%%%%%%%%%%%%%%%%%%%%%%%%%%%%%%%%%%%%%%%%%%%%%%%%
%%%%%%%%%%%%%%%%%%%%%%%%%%%%%%%%%%%%
%SECTION VI: CONCLUDING REMARKS AND DISCUSSION
%%%%%%%%%%%%%%%%%%%%%%%%%%%%%%%%%%%%
\section{Concluding Remarks and Discussion}
In most CPS applications there is  a large volume of data of a dynamic nature which is added to the system gradually in real time and  not altogether in advance.
As the volume of data  increases, the domain of the control strategies also increases, and thus it becomes challenging to search for an optimal strategy. Even if an optimal strategy is found, implementing such strategies with increasing domains is burdensome.
In  such CPS applications, we typically assume an ideal model of the system which is used to derive the optimal control strategy. Such  model-based control approaches cannot effectively facilitate optimal solutions with performance guarantees due to the discrepancy between the model and the actual CPS. On the other hand, traditional supervised learning approaches cannot always  facilitate robust solutions using data derived offline. By contrast, applying reinforcement learning approaches directly to the actual CPS might impose significant implications on safety and robust operation of the system. 

In this paper, we presented a theoretical framework that circumvents these challenges. The framework can combine offline model-based control with online learning approaches to yield the optimal control strategy for the system.
There are two features which sharply distinguish the framework presented here from previous learning-based, or combined learning and control approaches reported in the literature to date: (1)  the CPS  imposes a nonclassical information structure  while the state of the system is not fully observed; and (2) the large volume of data that is added to the system gradually is compressed to a sufficient information state without loss of optimality  that takes values in a time-invariant space. Therefore, the  volume of data which is added to the system gradually does not lead the domain of the control strategies to increase with time. 

In our exposition, we restricted attention to centralized strategies. Ongoing research includes expanding the framework to decentralized strategies.  A direction of future research should consider investigating how potential errors in the communication between the subsystems could be addressed.

%%%%%%%%%%%%%%%%%%%%%%%%%%%%%%%%%%%%%%%%%%%%%%%
%%%%%%%%%%%%%%%%%%%%%%%%%%%%%%%%%%%%%%%%%%%%%%%
%%%%%%%Appendices
%%%%%%%%%%%%%%%%%%%%%%%%%%%%%%%%%%%%%%%%%%%%%%%
%%%%%%%%%%%%%%%%%%%%%%%%%%%%%%%%%%%%%%%%%%%%%%%

\appendix

%%%%%%%%%%%%%%%%%%%%%%%%%%%%%%%%%%%%%%%%%%%%%%
%%%%%%%%%%%%%%%%%%%%%%%%%%%%%%%%%%%%%%%%%%%%%%

\section{Proof of Theorem \ref{theo:y_t}}\label{app:1}
By applying Bayes' rule, we have
\begin{align}
	&p^{\textbf{g}}(X_{t+1},\hat{X}_{t+1}~|~\Delta_{t+1}, \Lambda^{1:K}_{t+1})\nonumber\\
	&=\frac{\splitfrac{p^{\textbf{g}}(Y^{1:K}_{t+1}~|~X_{t+1},\hat{X}_{t+1}, \Delta_{t+1}, \Lambda^{1:K}_t, U^{1:K}_t)}{\cdot p^{\textbf{g}}(X_{t+1}, \hat{X}_{t+1}, \Delta_{t+1}, \Lambda^{1:K}_t, U^{1:K}_t)}}{p^{\textbf{g}}(\Delta_{t+1}, \Lambda^{1:K}_{t+1})}
\end{align}	
\begin{align}
	&= \frac{p(Y^{1:K}_{t+1}~|~X_{t+1}) ~p^{\textbf{g}}(X_{t+1}, \hat{X}_{t+1},	\Delta_{t+1}, \Lambda^{1:K}_t, U^{1:K}_t)}{p^{\textbf{g}}(\Delta_{t+1}, \Lambda^{1:K}_{t+1})}\nonumber\\
	&= \frac{\splitfrac{p(Y^{1:K}_{t+1}~|~X_{t+1}) ~p^{\textbf{g}}(X_{t+1},\hat{X}_{t+1} ~|~ \Delta_{t+1}, \Lambda^{1:K}_t,  U^{1:K}_t )}{\cdot p^{\textbf{g}}(\Delta_{t+1}, \Lambda^{1:K}_t,  U^{1:K}_t )}}{p^{\textbf{g}}(\Delta_{t+1}, \Lambda^{1:K}_{t+1})},	\label{eq:theo1a}
\end{align}
where in the second equality we used Lemma \ref{lem:y_t}. 

Next, 
\begin{multline} 
	p^{\textbf{g}}(\Delta_{t+1}, \Lambda^{1:K}_{t+1})= p^{\textbf{g}}(\Delta_{t+1}, \Lambda^{1:K}_t, Y^{1:K}_{t+1}, U^{1:K}_t)\nonumber\\
	=\int_{\mathscr{X}_{t+1}} \int_{\mathscr{X}_{t+1}} p^{\textbf{g}}(X_{t+1},\hat{X}_{t+1}, \Delta_{t+1}, \Lambda^{1:K}_t, Y^{1:K}_{t+1}, \nonumber\\U^{1:K}_t )~dX_{t+1}~d\hat{X}_{t+1} \nonumber\\
	=\int_{\mathscr{X}_{t+1}}\int_{\mathscr{X}_{t+1}} p^{\textbf{g}}(Y^{1:K}_{t+1}~|~X_{t+1}, \hat{X}_{t+1}, \Delta_{t+1}, \Lambda^{1:K}_t, U^{1:K}_t)\nonumber\\
	\cdot p^{\textbf{g}}(X_{t+1}, \hat{X}_{t+1}, \Delta_{t+1}, \Lambda^{1:K}_t,  U^{1:K}_t )~dX_{t+1}~d\hat{X}_{t+1} \\
	= \int_{\mathscr{X}_{t+1}}\int_{\mathscr{X}_{t+1}} p^{\textbf{g}}(Y^{1:K}_{t+1}~|~X_{t+1}, \hat{X}_{t+1}, \Delta_{t+1}, \Lambda^{1:K}_t, U^{1:K}_t)\nonumber\\
	\cdot p^{\textbf{g}}(X_{t+1},\hat{X}_{t+1} ~|~ \Delta_{t+1}, \Lambda^{1:K}_t,  U^{1:K}_t ) \nonumber \\
	\cdot p^{\textbf{g}}(\Delta_{t+1}, \Lambda^{1:K}_t,  U^{1:K}_t )~dX_{t+1}~d\hat{X}_{t+1},\label{eq:theo1b}
\end{multline}	
where by Lemma \ref{lem:y_t}, the last equation becomes
\begin{gather}
	p^{\textbf{g}}(\Delta_{t+1}, \Lambda^{1:K}_{t+1})\nonumber\\
	=\int_{\mathscr{X}_{t+1}}\int_{\mathscr{X}_{t+1}} p(Y^{1:K}_{t+1}~|~X_{t+1})~p^{\textbf{g}}(X_{t+1},\hat{X}_{t+1} ~|~ \Delta_{t+1},  \nonumber\\
	\Lambda^{1:K}_t, U^{1:K}_t )\cdot p^{\textbf{g}}(\Delta_{t+1}, \Lambda^{1:K}_t,  U^{1:K}_t )~dX_{t+1}~d\hat{X}_{t+1}.\label{eq:theo1ca}
\end{gather}
Note that $p^{\textbf{g}}(X_{t+1}, \hat{X}_{t+1} ~|~ \Delta_{t+1},\Lambda^{1:K}_t, U^{1:K}_t ) =p^{\textbf{g}}(X_{t+1},\hat{X}_{t+1} ~|~$ $\Delta_{t},$ $ \Lambda^{1:K}_t,  U^{1:K}_t )$ since $Y^{1:K}_{t-n+1}$ and $U^{1:K}_{t-n+1}$ are already included in $\Lambda^{1:K}_{t},$ hence we can write \eqref{eq:theo1ca} as 
\begin{gather}
	p^{\textbf{g}}(\Delta_{t+1}, \Lambda^{1:K}_{t+1})\nonumber\\
	=\int_{\mathscr{X}_{t+1}}\int_{\mathscr{X}_{t+1}} p(Y^{1:K}_{t+1}~|~X_{t+1})~p^{\textbf{g}}(X_{t+1}, \hat{X}_{t+1} ~|~\Delta_{t},  \nonumber\\
	\Lambda^{1:K}_t,  U^{1:K}_t )\cdot p^{\textbf{g}}(\Delta_{t+1}, \Lambda^{1:K}_t,  U^{1:K}_t )~dX_{t+1}~d\hat{X}_{t+1}.\label{eq:theo1c}
\end{gather}	  
Substituting \eqref{eq:theo1c} into \eqref{eq:theo1a}, we have
\begin{gather}
	p^{\textbf{g}}(X_{t+1}, \hat{X}_{t+1} ~|~\Delta_{t+1}, \Lambda^{1:K}_{t+1})\nonumber\\
	\tiny = \frac{p(Y^{1:K}_{t+1}~|~X_{t+1}) ~p^{\textbf{g}}(X_{t+1},\hat{X}_{t+1} ~|~ \Delta_{t}, \Lambda^{1:K}_t,  U^{1:K}_t )}{\splitfrac{\int_{\mathscr{X}_{t+1}}\int_{\mathscr{X}_{t+1}} p(Y^{1:K}_{t+1}~|~X_{t+1})~p^{\textbf{g}}(X_{t+1}, \hat{X}_{t+1} ~|~\Delta_{t}, }  {\Lambda^{1:K}_t,  U^{1:K}_t )~dX_{t+1}~d\hat{X}_{t+1} }}. \label{eq:theo1h}
\end{gather}

Next,
\begin{align}
	p^{\textbf{g}}(X_{t+1}, \hat{X}_{t+1} ~|~ \Delta_{t}, \Lambda^{1:K}_t,  U^{1:K}_t )&\nonumber\\
	= \int_{\mathscr{X}_{t}}\int_{\mathscr{X}_{t}} p^{\textbf{g}}(X_{t+1}, \hat{X}_{t+1} ~|~ X_{t}, \hat{X}_{t}, \Delta_{t}, \Lambda^{1:K}_t,  &U^{1:K}_t )\nonumber\\
	\cdot p^{\textbf{g}}(X_{t}, \hat{X}_{t} ~|~ \Delta_{t}, \Lambda^{1:K}_t,  U^{1:K}_t )~dX_{t}~d\hat{X}_{t}&.
	\label{eq:theo1d}
\end{align}	
By Lemma \ref{lem:x_t1ut} and Remark \ref{cor:lemU}, \eqref{eq:theo1d} becomes
\begin{align}
	p^{\textbf{g}}(X_{t+1}, \hat{X}_{t+1} ~|~ \Delta_{t}, \Lambda^{1:K}_t,  U^{1:K}_t )&\nonumber\\
	= \int_{\mathscr{X}_{t}}\int_{\mathscr{X}_{t}} p(X_{t+1}, \hat{X}_{t+1} ~|~ X_t, \hat{X}_{t}, U^{1:K}_t )&\nonumber\\
	\cdot p(X_{t}, \hat{X}_{t} ~|~ \Delta_{t}, \Lambda^{1:K}_t)~dX_{t}~d\hat{X}_{t}.&
	\label{eq:theo1e}
\end{align}		

Substituting \eqref{eq:theo1e} into \eqref{eq:theo1h} yields
\begin{gather}
	p^{\textbf{g}}(X_{t+1}, \hat{X}_{t+1} ~|~\Delta_{t+1}, \Lambda^{1:K}_{t+1})\nonumber\\
	\tiny = \frac{\splitfrac{p(Y^{1:K}_{t+1}~|~X_{t+1})~	
			\int_{\mathscr{X}_{t}}\int_{\mathscr{X}_{t}} p(X_{t+1}, \hat{X}_{t+1} ~|~ X_t, \hat{X}_{t}, U^{1:K}_t )}{\cdot p(X_{t}, \hat{X}_{t} ~|~ \Delta_{t}, \Lambda^{1:K}_t)~dX_{t}~d\hat{X}_{t}}}{\splitfrac{\int_{\mathscr{X}_{t+1}}\int_{\mathscr{X}_{t+1}} p(Y^{1:K}_{t+1}~|~X_{t+1})~\int_{\mathscr{X}_{t}}\int_{\mathscr{X}_{t}} p(X_{t+1}, \hat{X}_{t+1} ~|~ X_t, }  { U^{1:K}_t )~p(X_{t}, \hat{X}_{t} ~|~ \Delta_{t}, \Lambda^{1:K}_t)~dX_{t}~d\hat{X}_{t}~dX_{t+1}~d\hat{X}_{t+1} }}. \label{eq:theo1j}
\end{gather}

Therefore,  $p^{\textbf{g}}(X_{t+1},\hat{X}_{t+1} ~|~\Delta_{t+1}, \Lambda^{1:K}_{t+1})$ does not depend on the control strategy $\textbf{g}$, so we can drop the superscript. Moreover, we can choose appropriate function $\phi_t$ such that

\begin{align}
	&p^{\textbf{g}}(X_{t+1},\hat{X}_{t+1} ~|~\Delta_{t+1}, \Lambda^{1:K}_{t+1})\nonumber\\
	&= \Pi_{t+1}(\Delta_{t+1}, \Lambda^{1:K}_{t+1})(X_{t+1},\hat{X}_{t+1}) \nonumber\\
	&= \phi_t\left[
	\Pi_{t}(\Delta_{t}, \Lambda^{1:K}_{t})(X_{t}, \hat{X}_{t}),  Y^{1:K}_{t+1}, U^{1:K}_t \right]. \label{eq:theo1k}
\end{align}

%%%%%%%%%%%%%%%%%%%%%%%%%%%%%%%%%%%%%%%%%%%%%%
%%%%%%%%%%%%%%%%%%%%%%%%%%%%%%%%%%%%%%%%%%%%%%

\section{Proof of Theorem \ref{theo:dp}}\label{app:2}
(a) We prove \eqref{theo2:1c} by induction. For $t=T$,
\begin{align}			
	J_T(\textbf{g};\hat{x}_{T})\coloneqq& ~\mathbb{E}^{\textbf{g}}\Big[c_T(X_T) |~\Delta_{T}, \Lambda^{1:K}_{T} \Big] \nonumber\\
	= \int_{\mathscr{X}_T} c_T(X_T) ~&\Pi_{T}(\Delta_{T}, \Lambda^{1:K}_{T})(X_T,\hat{X}_{T})~ dX_T,
	\label{theo2:1d}	
\end{align}	
and so \eqref{theo2:1c} holds with equality. 
Suppose that \eqref{theo2:1c} holds for $t+1$. Then,

\begin{align}			
	&J_t(\textbf{g};\hat{x}_{t:T})= \mathbb{E}^{\textbf{g}}\Bigg[\sum_{l=t}^{T-1}\Big[c_l(X_l,U^{1:K}_l)
	+ \beta \cdot|X_{l+1} - \hat{X}_{l+1}|^2 \Big]\nonumber\\&+c_T(X_T) ~|~\Delta_{t}, \Lambda^{1:K}_{t} \Bigg]\nonumber\\
	&= \mathbb{E}^{\textbf{g}}\Bigg[c_t(X_t,U^{1:K}_t) +\beta \cdot |X_{t+1}- \hat{X}_{t+1}|^2 \nonumber\\
	&+\sum_{l=t+1}^{T-1}\Big[c_l(X_l,U^{1:K}_l)+\beta \cdot|X_{l+1} - \hat{X}_{l+1}|^2~\Big] \nonumber\\&+ c_T(X_T)~ |~\Delta_{t}, \Lambda^{1:K}_{t} \Bigg]\nonumber\\
	&=  \mathbb{E}^{\textbf{g}}\Bigg[ \mathbb{E}^{\textbf{g}}  \bigg[c_t(X_t,U^{1:K}_t) + \beta\cdot |X_{t+1}  - \hat{X}_{t+1}|^2\nonumber\\&+\sum_{l=t+1}^{T-1}\Big[c_l(X_l,U^{1:K}_l) + \beta\cdot |X_{l+1}  - \hat{X}_{l+1}|^2~ \Big]\nonumber\\&+ c_T(X_T)~|~\Delta_{t}, \Lambda^{1:K}_{t}, U^{1:K}_{t} \bigg] 
	~|~\Delta_{t}, \Lambda^{1:K}_{t}\Bigg] \nonumber\\				
	&\ge  \mathbb{E}^{\textbf{g}}\bigg[ \mathbb{E}^{\textbf{g}}  \Big[c_t(X_t,U^{1:K}_t)+ \beta\cdot |X_{t+1}  - \hat{X}_{t+1}|^2~\nonumber\\& +V_{t+1}\big( \phi_t\big[ \Pi_{t}(\Delta_{t}, \Lambda^{1:K}_{t})(X_{t},\hat{X}_{t}), Y^{1:K}_{t+1}, U^{1:K}_t \big]\big) ~|\nonumber\\
	&~\Pi_{t}(\Delta_{t}, \Lambda^{1:K}_{t}), U^{1:K}_{t} \Big] ~|~\Delta_{t}, \Lambda^{1:K}_{t}\bigg]\nonumber\\
	&=  \mathbb{E}^{\textbf{g}}\bigg[ V_{t}\big(  \Pi_{t}(\Delta_{t}, \Lambda^{1:K}_{t})(X_{t},\hat{X}_{t})\big) ~ |~\Delta_{t}, \Lambda^{1:K}_{t}\bigg]\nonumber\\&=V_{t}\big(  \Pi_{t}(\Delta_{t}, \Lambda^{1:K}_{t})(X_{t},\hat{X}_{t})\big), 	
	\label{theo2:1e}	
\end{align}	
where, in the inequality, we used the hypothesis and, in the last equality, we used \eqref{theo2:1b}. Thus, \eqref{theo2:1c} holds for all $t$.

(b) We prove the second part of the theorem by induction too. For $t=T$,
\begin{align}			
	J_T(\textbf{g};\hat{x}_{T})\coloneqq& ~\mathbb{E}^{\textbf{g}}\Big[c_T(X_T) |~\Delta_{T}, \Lambda^{1:K}_{T} \Big] \nonumber\\
	= \int_{\hat{\mathscr{X}}_T} c_T(X_T) ~&\Pi_{T}(\Delta_{T}, \Lambda^{1:K}_{T})(X_T,\hat{X}_{T})~dX_T.
	\label{theo3:1c}	
\end{align}	

Suppose that \eqref{theo2:1b} holds for $t+1$. Then
\begin{align}			
	&\inf_{u^{1:K}_t\in\prod_{k\in\mathcal{K}} \mathcal{U}_t^k }\mathbb{E}^{\textbf{g}}\Bigg[\sum_{l=t}^{T-1}\Big[c_l(X_l,U^{1:K}_l)
	+ \beta\cdot|X_{l+1}- \hat{X}_{l+1}|^2~ \Big]\nonumber\\&+c_T(X_T) ~|~\Delta_{t}, \Lambda^{1:K}_{t} \Bigg]
\end{align}	
\begin{align}
	&= \inf_{u^{1:K}_t\in\prod_{k\in\mathcal{K}} \mathcal{U}_t^k }\mathbb{E}^{\textbf{g}}\Bigg[ c_t(X_t,U^{1:K}_t) + \beta \cdot |X_{t+1} - \hat{X}_{t+1}|^2 \nonumber\\& +\sum_{l=t+1}^{T-1}\Big[c_l(X_l,U^{1:K}_l)+\beta \cdot |X_{l+1}  - \hat{X}_{l+1}|^2\Big]\nonumber\\&+ c_T(X_T)	~|~\Delta_{t}, \Lambda^{1:K}_{t}\Bigg]\nonumber\\
	&= \inf_{u^{1:K}_t\in\prod_{k\in\mathcal{K}} \mathcal{U}_t^k }\mathbb{E}^{\textbf{g}}\Bigg[ \mathbb{E}^{\textbf{g}}  \bigg[c_t(X_t,U^{1:K}_t) + \beta\cdot  |X_{t+1} - \hat{X}_{t+1}|^2 \nonumber\\ 
	&+\sum_{l=t+1}^{T-1}\Big[c_l(X_l,U^{1:K}_l)+ \beta\cdot |X_{l+1} - \hat{X}_{l+1}|^2~ \Big]\nonumber\\&+ c_T(X_T)|~\Delta_{t}, \Lambda^{1:K}_{t}, U^{1:K}_{t} \bigg] 
	~|~\Delta_{t}, \Lambda^{1:K}_{t}\Bigg]\nonumber\\
	&= \inf_{u^{1:K}_t\in\prod_{k\in\mathcal{K}} \mathcal{U}_t^k } \mathbb{E}^{\textbf{g}}\bigg[ \mathbb{E}^{\textbf{g}}  \Big[c_t(X_t,U^{1:K}_t)+ \beta\cdot|X_{t+1} - \hat{X}_{t+1}|^2~\nonumber\\& +V_{t+1}\big( \phi_t\big[ \Pi_{t}(\Delta_{t}, \Lambda^{1:K}_{t})(X_{t}, \hat{X}_{t}), Y^{1:K}_{t+1}, U^{1:K}_t \big]\big) ~|~\Pi_{t}(\Delta_{t},\nonumber\\
	& \Lambda^{1:K}_{t}), U^{1:K}_{t} \Big] |~\Delta_{t}, \Lambda^{1:K}_{t}\bigg]\nonumber\\
	&= \mathbb{E}^{\textbf{g}}\bigg[ V_{t}\big(  \Pi_{t}(\Delta_{t}, \Lambda^{1:K}_{t})(X_{t},\hat{X}_{t})\big)  ~|~\Delta_{t}, \Lambda^{1:K}_{t}\bigg] \nonumber\\&=V_{t}\big(  \Pi_{t}(\Delta_{t}, \Lambda^{1:K}_{t})(X_{t},\hat{X}_{t})\big), 	
	\label{theo3:1d}	
\end{align}	
where, in the third equality, we used the hypothesis, and in the forth equality, $u^{1:K}_t$ achieves the infimum. Thus, \eqref{theo2:1b} holds for all $t$.

For $t=0$, \eqref{theo2:1c} yields $J_0(\textbf{g}; \hat{x}_{0:T})=V_0\big(\Pi_{0}(\Delta_{0}, \Lambda^{1:K}_{0})$ $(X_{0},\hat{X}_{0})\big)$. Taking expectations
\begin{align}			
	J(\textbf{g}; \hat{x}_{0:T})= \mathbb{E}^{\textbf{g}}\Big[V_0\big(\Pi_{0}(\Delta_{0}, \Lambda^{1:K}_{0})(X_{0},\hat{X}_{0})\big)  \Big].
	\label{theo3:1e}	
\end{align}	
By  \eqref{theo2:1c}, it follows that for any other $\textbf{g}'\in\mathcal{G}$,  
\begin{align}			
	J(\textbf{g}';\hat{x}_{0:T})\ge \mathbb{E}^{\textbf{g}}\Big[V_0\big(\Pi_{0}(\Delta_{0}, \Lambda^{1:K}_{0})(X_{0},\hat{X}_{0})\big)  \Big].
	\label{theo3:1f}	
\end{align}

%%%%%%%%%%%%%%%%%%%%%%%%%%%%%%%%%%%%%%%%%%%%%%
%%%%%%%%%%%%%%%%%%%%%%%%%%%%%%%%%%%%%%%%%%%%%%

\section{Proof of Lemma \ref{lemma:homo}}\label{app:homo}
Obviously, for $t=T$,
\begin{align}			
	&V_T\big(\rho~ \Pi_{T}(\Delta_{T}, \Lambda^{1:K}_{T})\big)\nonumber\\
	&=  \int_{\mathscr{X}_T} \int_{\mathscr{X}_T}c_T(X_T) ~\rho ~\Pi_{T}(\Delta_{T}, \Lambda^{1:K}_{T})(X_T,\hat{X}_t)~ dX_T~d\hat{X}_t\nonumber\\
	& =\rho ~V_T\big(\Pi_{T}(\Delta_{T}, \Lambda^{1:K}_{T})\big).
	\label{theo_concave:1c}	
\end{align}	
For $t=0,\ldots,T-1$, by assigning $\Pi_{t}=\rho~ \Pi_{t}$ [recall $p(X_t,\hat{X}_t ~|~ \Delta_{t}, \Lambda^{1:K}_t)=\Pi_{t}(\Delta_{t}, \Lambda^{1:K}_{t})$], \eqref{theo2:1b} becomes
\begin{align}			
	&V_t\big(\rho ~\Pi_{t}(\Delta_{t}, \Lambda^{1:K}_{t})\big)\nonumber\\
	&= \inf_{u^{1:K}_t\in\prod_{k\in\mathcal{K}} \mathcal{U}_t^k }\bigg[ \int_{\mathscr{X}_t} \int_{\mathscr{X}_t} c_t(X_t,U^{1:K}_t) ~\rho~ \nonumber\\
	&\cdot \Pi_{t}(\Delta_{t}, \Lambda^{1:K}_{t})(X_t,\hat{X}_t)~ dX_t~d\hat{X}_t\nonumber\\
	&+ \int_{\mathscr{Y}_{t+1}} \int_{\mathscr{X}_{t+1}} \int_{\mathscr{X}_{t+1}} \int_{\mathscr{X}_{t}} \int_{\mathscr{X}_{t}} V_{t+1}\big(\rho~ \Pi_{t+1}(\Delta_{t+1}, \Lambda^{1:K}_{t+1})\big) \nonumber\\
	& \cdot p(Y^{1:K}_{t+1}~|~X_{t+1})~ p(X_{t+1},\hat{X}_{t+1} ~|~ X_t, \hat{X}_t,U^{1:K}_t )~\rho \nonumber\\ \cdot &p(X_{t},\hat{X}_t ~|~ \Delta_{t}, \Lambda^{1:K}_t) dX_{t}~d\hat{X}_t~dX_{t+1}~d\hat{X}_{t+1}~dY^{1:K}_{t+1}\bigg],
	\label{theo_concave:1d}	
\end{align}		
where $\mathscr{Y}_{t+1}=\otimes_{k\in\mathcal{K}}\mathscr{Y}^k$.

Next, from \eqref{eq:xt1}, 
\begin{align}
	&\rho~\Pi_{t+1}(\Delta_{t+1}, \Lambda^{1:K}_{t+1}) \nonumber\\
	&= \frac{\splitfrac{ p(Y^{1:K}_{t+1}~|~X_{t+1}) ~\int_{\mathscr{X}_{t}} ~\int_{\mathscr{X}_{t}} p(X_{t+1},\hat{X}_{t+1} ~|~ X_t, \hat{X}_t,} {U^{1:K}_t ) \cdot \bcancel{\rho} \cdot p(X_{t},\hat{X}_t ~|~ \Delta_{t},\Lambda^{1:K}_t)~dX_{t} ~d\hat{X}_t}}{\splitfrac{ \int_{\mathscr{X}_{t+1}} \int_{\mathscr{X}_{t+1}}  p(Y^{1:K}_{t+1}~|~X_{t+1})~	\int_{\mathscr{X}_{t}} \int_{\mathscr{X}_{t}} p(X_{t+1},\hat{X}_{t+1} ~| 
		}  {X_t, U^{1:K}_t ) \cdot \bcancel{\rho}~p(X_{t},\hat{X}_t ~|~ \Delta_{t},\Lambda^{1:K}_t)~dX_{t}~d\hat{X}_t~dX_{t+1}}}, \nonumber\\
	&= \Pi_{t+1}(\Delta_{t+1}, \Lambda^{1:K}_{t+1}).\label{theo_concave:1e}
\end{align} 
Substituting \eqref{theo_concave:1e} into \eqref{theo_concave:1d}, we have
\begin{align}			
	&V_t\big(\rho ~\Pi_{t}(\Delta_{t}, \Lambda^{1:K}_{t})\big)\nonumber\\
	&= \inf_{u^{1:K}_t\in\prod_{k\in\mathcal{K}} \mathcal{U}_t^k }\bigg[ \int_{\mathscr{X}_t} \int_{\mathscr{X}_t} c_t(X_t,U^{1:K}_t) ~\rho~ \nonumber\\
	&\cdot \Pi_{t}(\Delta_{t}, \Lambda^{1:K}_{t})(X_t,\hat{X}_t)~ dX_t~d\hat{X}_t\nonumber\\
	&+ \int_{\mathscr{Y}_{t+1}} \int_{\mathscr{X}_{t+1}} \int_{\mathscr{X}_{t+1}} \int_{\mathscr{X}_{t}} \int_{\mathscr{X}_{t}} V_{t+1}\big(\Pi_{t+1}(\Delta_{t+1}, \Lambda^{1:K}_{t+1})\big) \nonumber\\
	& \cdot p(Y^{1:K}_{t+1}~|~X_{t+1})~ p(X_{t+1},\hat{X}_{t+1} ~|~ X_t, \hat{X}_t,U^{1:K}_t )~\rho \nonumber\\ \cdot &p(X_{t},\hat{X}_t ~|~ \Delta_{t}, \Lambda^{1:K}_t) dX_{t}~d\hat{X}_t~dX_{t+1}~d\hat{X}_{t+1}~dY^{1:K}_{t+1}\bigg]\nonumber\\
	&=\rho~V_t\big(\Pi_{t}(\Delta_{t}, \Lambda^{1:K}_{t})\big).
	\label{theo_concave:1f}	
\end{align}

%%%%%%%%%%%%%%%%%%%%%%%%%%%%%%%%%%%%%%%%%%%%%%
%%%%%%%%%%%%%%%%%%%%%%%%%%%%%%%%%%%%%%%%%%%%%%

\section{Proof of Theorem \ref{theo:concave}}\label{app:concave}
Starting with \eqref{theo2:1b}, we have
\begin{align}			
	&V_t\big(\Pi_{t}(\Delta_{t}, \Lambda^{1:K}_{t})\big)\nonumber\\
	&= \inf_{u^{1:K}_t\in\prod_{k\in\mathcal{K}} \mathcal{U}_t^k }\bigg[ \int_{\mathscr{X}_t} \int_{\mathscr{X}_t} c_t(X_t,U^{1:K}_t) ~
\end{align}	
\begin{align}
	&\cdot \Pi_{t}(\Delta_{t}, \Lambda^{1:K}_{t})(X_t,\hat{X}_t)~ dX_t~d\hat{X}_t\nonumber\\
	&+ \int_{\mathscr{Y}_{t+1}} \int_{\mathscr{X}_{t+1}} \int_{\mathscr{X}_{t+1}} \int_{\mathscr{X}_{t}} \int_{\mathscr{X}_{t}} V_{t+1}\big(\Pi_{t+1}(\Delta_{t+1}, \Lambda^{1:K}_{t+1})\big) \nonumber\\
	& \cdot p(Y^{1:K}_{t+1}~|~X_{t+1})~ p(X_{t+1},\hat{X}_{t+1} ~|~ X_t, \hat{X}_t,U^{1:K}_t ) \nonumber\\ \cdot &p(X_{t},\hat{X}_t ~|~ \Delta_{t}, \Lambda^{1:K}_t) dX_{t}~d\hat{X}_t~dX_{t+1}~d\hat{X}_{t+1}~dY^{1:K}_{t+1}\bigg],
	\label{theo_homo:1b}
\end{align}		
where $\mathscr{Y}_{t+1}=\otimes_{k\in\mathcal{K}}\mathscr{Y}^k$.

Choosing 
\begin{align}	
	&\rho =  \int_{\mathscr{X}_{t+1}} \int_{\mathscr{X}_{t+1}} \int_{\mathscr{X}_{t}}\int_{\mathscr{X}_{t}} p(Y^{1:K}_{t+1}~|~X_{t+1})~\nonumber\\
	&\cdot p(X_{t+1},\hat{X}_{t+1} ~|~ X_t,\hat{X}_{t}, U^{1:K}_t )\cdot p(X_{t},\hat{X}_{t} ~|~ \Delta_{t},\Lambda^{1:K}_t)~dX_{t}~\nonumber\\&\cdot d\hat{X}_{t}~dX_{t+1}~d\hat{X}_{t+1},\label{theo_concave:1rho}	
\end{align}
we can use the positive homogeneity of $V_t\big(\Pi_{t}(\Delta_{t}, \Lambda^{1:K}_{t})\big)$ (Lemma \ref{lemma:homo}) to write the second part of \eqref{theo_homo:1b} as follows
\begin{align}			
	&\int_{\mathscr{Y}_{t+1}} \int_{\mathscr{X}_{t+1}} \int_{\mathscr{X}_{t+1}} \int_{\mathscr{X}_{t}} \int_{\mathscr{X}_{t}} V_{t+1}\big(\Pi_{t+1}(\Delta_{t+1}, \Lambda^{1:K}_{t+1})\big) \nonumber\\
	& \cdot p(Y^{1:K}_{t+1}~|~X_{t+1})~ p(X_{t+1},\hat{X}_{t+1} ~|~ X_t, \hat{X}_t,U^{1:K}_t ) \nonumber\\ \cdot &p(X_{t},\hat{X}_t ~|~ \Delta_{t}, \Lambda^{1:K}_t) dX_{t}~d\hat{X}_t~dX_{t+1}~d\hat{X}_{t+1}~dY^{1:K}_{t+1}\nonumber
\end{align}	
\begin{align}
	&=\int_{\mathscr{Y}_{t+1}} V_{t+1}\big(\rho~ \Pi_{t+1}(\Delta_{t+1}, \Lambda^{1:K}_{t+1})\big)~dY^{1:K}_{t+1}\nonumber\\
	&=\int_{\mathscr{Y}_{t+1}} V_{t+1}\bigg(p(Y^{1:K}_{t+1}~|~X_{t+1}) ~\int_{\mathscr{X}_{t}} ~\int_{\mathscr{X}_{t}} p(X_{t+1},\nonumber\\
	&\hat{X}_{t+1} ~|~ X_t, \hat{X}_t, U^{1:K}_t ) \cdot  p(X_{t},\hat{X}_t ~|~ \Delta_{t},\Lambda^{1:K}_t)~dX_{t} \nonumber\\
	&~d\hat{X}_t	\bigg)~dY^{1:K}_{t+1},
	\label{theo_concave:1h}	
\end{align}
where, in the last equality, we substituted \eqref{theo_concave:1rho} and \eqref{eq:theo1j}.

Thus, we can write \eqref{theo_homo:1b} as
\begin{align}			
	&V_t\big(\Pi_{t}(\Delta_{t}, \Lambda^{1:K}_{t})\big)\nonumber\\
	&=  \inf_{u^{1:K}_t\in\prod_{k\in\mathcal{K}} \mathcal{U}_t^k }\bigg[ \int_{\mathscr{X}_t} \int_{\mathscr{X}_t} c_t(X_t,U^{1:K}_t) ~ \nonumber\\
	&\cdot \Pi_{t}(\Delta_{t}, \Lambda^{1:K}_{t})(X_t,\hat{X}_t)~ dX_t~d\hat{X}_t\nonumber\\
	&+\int_{\mathscr{Y}_{t+1}} V_{t+1}\bigg(p(Y^{1:K}_{t+1}~|~X_{t+1}) ~\int_{\mathscr{X}_{t}} ~\int_{\mathscr{X}_{t}} p(X_{t+1},\nonumber\\
	&\hat{X}_{t+1} ~|~ X_t, \hat{X}_t, U^{1:K}_t ) \cdot  p(X_{t},\hat{X}_t ~|~ \Delta_{t},\Lambda^{1:K}_t)~dX_{t} \nonumber\\
	&~d\hat{X}_t	\bigg)~dY^{1:K}_{t+1} \bigg].
	\label{theo_concave:1hh}	
\end{align}		

The remainder of the proof follows by induction. 
Suppose that $V_{t+1}\big( \Pi_{t+1}(\Delta_{t+1}, \Lambda^{1:K}_{t+1})\big)$ is concave. Since
\begin{align}
	&V_{t+1}\bigg(p(Y^{1:K}_{t+1}~|~X_{t+1}) ~\int_{\mathscr{X}_{t}} ~\int_{\mathscr{X}_{t}} p(X_{t+1},\nonumber\\
	&\hat{X}_{t+1} ~|~ X_t, \hat{X}_t, U^{1:K}_t ) \cdot  p(X_{t},\hat{X}_t ~|~ \Delta_{t},\Lambda^{1:K}_t)~dX_{t}~d\hat{X}_t	\bigg),
	\label{theo_concave:1i}	
\end{align}
is the composition of a concave function and increasing linear function, it follows that it is concave. However, concavity is preserved by integration (see \cite{boyd2004}, p. 79), hence 
\begin{align}			
	&\int_{\mathscr{Y}_{t+1}} V_{t+1}\bigg(p(Y^{1:K}_{t+1}~|~X_{t+1}) ~\int_{\mathscr{X}_{t}} ~\int_{\mathscr{X}_{t}} p(X_{t+1},\nonumber\\
	&\hat{X}_{t+1} ~|~ X_t, \hat{X}_t, U^{1:K}_t ) \cdot  p(X_{t},\hat{X}_t ~|~ \Delta_{t},\Lambda^{1:K}_t)~dX_{t} \nonumber\\
	&~d\hat{X}_t	\bigg)~dY^{1:K}_{t+1}.
	\label{theo_concave:1j}	
\end{align}
is concave.
Since the pointwise infimum of concave functions is concave, \eqref{theo_concave:1hh} is concave.

\bibliographystyle{abbrvnat}
\bibliography{TAC_Ref_structure, TAC_Ref_IDS, TAC_Ref_Andreas,TAC2_learn,TAC_learn_Andreas}

\begin{thebibliography}{65}
\providecommand{\natexlab}[1]{#1}
\providecommand{\url}[1]{\texttt{#1}}
\expandafter\ifx\csname urlstyle\endcsname\relax
  \providecommand{\doi}[1]{doi: #1}\else
  \providecommand{\doi}{doi: \begingroup \urlstyle{rm}\Url}\fi

\bibitem[Aicardi et~al.(1987)Aicardi, Davoli, and
  Minciardi]{aicardi1987decentralized}
M.~Aicardi, F.~Davoli, and R.~Minciardi.
\newblock Decentralized optimal control of markov chains with a common past
  information set.
\newblock \emph{IEEE Transactions on Automatic Control}, 32\penalty0
  (11):\penalty0 1028--1031, 1987.

\bibitem[Akametalu et~al.(2014)Akametalu, Fisac, Gillula, Kaynama, Zeilinger,
  and Tomlin]{Akametalu:2014th}
A.~K. Akametalu, J.~F. Fisac, J.~H. Gillula, S.~Kaynama, M.~N. Zeilinger, and
  C.~J. Tomlin.
\newblock Reachability-based safe learning with gaussian processes.
\newblock In \emph{53rd IEEE Conference on Decision and Control}, pages
  1424--1431, 2014.

\bibitem[Armstrong et~al.(2021)Armstrong, Johnson, and
  Alleyne]{Armstrong:2021vw}
A.~A. Armstrong, A.~J.~W. Johnson, and A.~G. Alleyne.
\newblock An improved approach to iterative learning control for uncertain
  systems.
\newblock \emph{IEEE Transactions on Control Systems Technology}, 29\penalty0
  (2):\penalty0 546--555, 2021.

\bibitem[Arslan and Y{\"u}ksel(2017)]{Arslan:2017vo}
G.~Arslan and S.~Y{\"u}ksel.
\newblock Decentralized q-learning for stochastic teams and games.
\newblock \emph{IEEE Transactions on Automatic Control}, 62\penalty0
  (4):\penalty0 1545--1558, 2017.

\bibitem[{\AA}str{\"o}m and Wittenmark(1995)]{strm1989AdaptiveC}
K.~{\AA}str{\"o}m and B.~Wittenmark.
\newblock \emph{Adaptive Control}.
\newblock Addison-Wesley Publising Company, 1995.

\bibitem[Aswani et~al.(2013)Aswani, Gonzalez, Sastry, and
  Tomlin]{Aswani:2013ue}
A.~Aswani, H.~Gonzalez, S.~S. Sastry, and C.~Tomlin.
\newblock Provably safe and robust learning-based model predictive control.
\newblock \emph{Automatica}, 49\penalty0 (5):\penalty0 1216--1226, 2013.

\bibitem[Bertsekas(2017)]{Bertsekas2017}
D.~Bertsekas.
\newblock \emph{Dynamic Programming and Optimal Control}.
\newblock Athena Scientific, 4th edition, 2017.

\bibitem[Bertsekas and Tsitsiklis(1996)]{Bertsekas1996}
D.~P. Bertsekas and J.~N. Tsitsiklis.
\newblock \emph{Neuro-Dynamic Programming}.
\newblock Athena Scientific, 1996.

\bibitem[Bismut(1973)]{Bismut:1973aa}
J.~Bismut.
\newblock An example of interaction between information and control: The
  transparency of a game.
\newblock \emph{IEEE Transactions on Automatic Control}, 18\penalty0
  (5):\penalty0 518--522, 1973.
\newblock \doi{10.1109/TAC.1973.1100388}.

\bibitem[Boyd and Vandenberghe(2004)]{boyd2004}
S.~Boyd and L.~Vandenberghe.
\newblock \emph{Convex Optimization}.
\newblock Cambridge University Press, 2004.

\bibitem[Brand(1999)]{Brand:1999aa}
M.~Brand.
\newblock Structure learning in conditional probability models via an entropic
  prior and parameter extinction.
\newblock \emph{Neural Computation}, 11\penalty0 (5):\penalty0 1155--1182,
  1999.
\newblock \doi{10.1162/089976699300016395}.

\bibitem[Chalaki et~al.(2020)Chalaki, Beaver, Remer, Jang, Vinitsky, Bayen, and
  Malikopoulos]{chalaki2020ICCA}
B.~Chalaki, L.~E. Beaver, B.~Remer, K.~Jang, E.~Vinitsky, A.~Bayen, and A.~A.
  Malikopoulos.
\newblock Zero-shot autonomous vehicle policy transfer: From simulation to
  real-world via adversarial learning.
\newblock In \emph{IEEE 16th International Conference on Control \& Automation
  (ICCA)}, pages 35--40, 2020.

\bibitem[Dave and Malikopoulos(2019)]{Dave2019a}
A.~Dave and A.~A. Malikopoulos.
\newblock {Decentralized Stochastic Control in Partially Nested Information
  Structures}.
\newblock In \emph{IFAC-PapersOnLine}, volume~52, pages 97--102, Chicago, IL,
  USA, 2019.

\bibitem[Dave and Malikopoulos(2020)]{Dave2020a}
A.~Dave and A.~A. Malikopoulos.
\newblock Structural results for decentralized stochastic control with a
  word-of-mouth communication.
\newblock In \emph{2020 American Control Conference (ACC)}, pages 2796--2801.
  IEEE, 2020.

\bibitem[Dydek et~al.(2013)Dydek, Annaswamy, and
  Lavretsky]{Dydek2013AdaptiveCO}
Z.~Dydek, A.~Annaswamy, and E.~Lavretsky.
\newblock Adaptive control of quadrotor uavs: A design trade study with flight
  evaluations.
\newblock \emph{IEEE Transactions on Control Systems Technology}, 21:\penalty0
  1400--1406, 2013.

\bibitem[Eisen et~al.(2018)Eisen, Gatsis, Pappas, and Ribeiro]{Eisen:2018ws}
M.~Eisen, K.~Gatsis, G.~J. Pappas, and A.~Ribeiro.
\newblock Learning in non-stationary wireless control systems via newton's
  method.
\newblock In \emph{2018 Annual American Control Conference (ACC)}, pages
  1410--1417, 2018.

\bibitem[Fisac et~al.(2019)Fisac, Akametalu, Zeilinger, Kaynama, Gillula, and
  Tomlin]{Fisac:2019wb}
J.~F. Fisac, A.~K. Akametalu, M.~N. Zeilinger, S.~Kaynama, J.~Gillula, and
  C.~J. Tomlin.
\newblock A general safety framework for learning-based control in uncertain
  robotic systems.
\newblock \emph{IEEE Transactions on Automatic Control}, 64\penalty0
  (7):\penalty0 2737--2752, 2019.

\bibitem[Gatsis and Pappas(2021)]{Gatsis:2021tt}
K.~Gatsis and G.~J. Pappas.
\newblock Statistical learning for analysis of networked control systems over
  unknown channels.
\newblock \emph{Automatica}, 125:\penalty0 109386, 2021.

\bibitem[Guha and Annaswamy(2021)]{Guha2021OnlinePF}
A.~Guha and A.~Annaswamy.
\newblock Online policies for real-time control using mrac-rl.
\newblock \emph{ArXiv}, abs/2103.16551, 2021.

\bibitem[Gupta et~al.(2015)Gupta, Y{\"u}ksel, Ba{\c s}ar, and
  Langbort]{Gupta:2015aa}
A.~Gupta, S.~Y{\"u}ksel, T.~Ba{\c s}ar, and C.~Langbort.
\newblock On the existence of optimal policies for a class of static and
  sequential dynamic teams.
\newblock \emph{SIAM Journal on Control and Optimization}, 53\penalty0
  (3):\penalty0 1681--1712, 2015.

\bibitem[Gyorfi and Kohler(2007)]{Gyorfi:2007aa}
L.~Gyorfi and M.~Kohler.
\newblock Nonparametric estimation of conditional distributions.
\newblock \emph{IEEE Transactions on Information Theory}, 53\penalty0
  (5):\penalty0 1872--1879, 2007.
\newblock \doi{10.1109/TIT.2007.894631}.

\bibitem[Howard(1960)]{Howard}
R.~A. Howard.
\newblock \emph{{Dynamic Programming and Markov Process}}.
\newblock The MIT Press, 1960.

\bibitem[Ioannou and Sun(1996)]{Ioannou2012RobustAC}
P.~A. Ioannou and J.~Sun.
\newblock \emph{Robust Adaptive Control}.
\newblock PTR Prentice-Hall, 1996.

\bibitem[Kara and Y{\"u}ksel(2018)]{Kara:2018vu}
A.~D. Kara and S.~Y{\"u}ksel.
\newblock Robustness to incorrect system models in stochastic control and
  application to data-driven learning.
\newblock In \emph{2018 IEEE Conference on Decision and Control (CDC)}, pages
  2753--2758, 2018.
\newblock ISBN 2576-2370.
\newblock \doi{10.1109/CDC.2018.8619684}.

\bibitem[Khong et~al.(2016{\natexlab{a}})Khong, Ne{\v s}i{\'c}, and
  Krsti{\'c}]{Khong:2016ug}
S.~Z. Khong, D.~Ne{\v s}i{\'c}, and M.~Krsti{\'c}.
\newblock Iterative learning control based on extremum seeking.
\newblock \emph{Automatica}, 66:\penalty0 238--245, 2016{\natexlab{a}}.

\bibitem[Khong et~al.(2016{\natexlab{b}})Khong, Ne{\v s}i{\'c}, and
  Krsti{\'c}]{Khong:2016ws}
S.~Z. Khong, D.~Ne{\v s}i{\'c}, and M.~Krsti{\'c}.
\newblock An extremum seeking approach to sampled-data iterative learning
  control of continuous-time nonlinear systems.
\newblock \emph{IFAC-PapersOnLine}, 49\penalty0 (18):\penalty0 962--967,
  2016{\natexlab{b}}.

\bibitem[Kiumarsi et~al.(2018)Kiumarsi, Vamvoudakis, Modares, and
  Lewis]{Kiumarsi:2018tq}
B.~Kiumarsi, K.~G. Vamvoudakis, H.~Modares, and F.~L. Lewis.
\newblock Optimal and autonomous control using reinforcement learning: A
  survey.
\newblock \emph{IEEE Transactions on Neural Networks and Learning Systems},
  29\penalty0 (6):\penalty0 2042--2062, 2018.

\bibitem[Krichene et~al.(2015)Krichene, Drigh{\`e}s, and
  Bayen]{Krichene:2015vx}
W.~Krichene, B.~Drigh{\`e}s, and A.~M. Bayen.
\newblock Online learning of nash equilibria in congestion games.
\newblock \emph{SIAM Journal on Control and Optimization}, 53\penalty0
  (2):\penalty0 1056--1081, 2015.

\bibitem[Krichene et~al.(2018)Krichene, Castillo, and Bayen]{Krichene:2018we}
W.~Krichene, M.~S. Castillo, and A.~Bayen.
\newblock On social optimal routing under selfish learning.
\newblock \emph{IEEE Transactions on Control of Network Systems}, 5\penalty0
  (1):\penalty0 479--488, 2018.

\bibitem[Krishnamurthy(2016)]{Krishnamurthy2016}
V.~Krishnamurthy.
\newblock \emph{Partially Observed Markov Decision Processes (From Filtering to
  Controlled Sensing)}.
\newblock Cambridge University Press, 2016.

\bibitem[Kumar and Varaiya(1986)]{Kumar1986}
P.~R. Kumar and P.~Varaiya.
\newblock \emph{Stochastic Systems: Estimation, Identification and Adaptive
  Control}.
\newblock Prentice-Hall, Inc., Upper Saddle River, NJ, USA, 1986.
\newblock ISBN 0-13-846684-X.

\bibitem[Kurtaran(1979)]{Kurtaran:1979aa}
B.~Kurtaran.
\newblock Corrections and extensions to "decentralized stochastic control with
  delayed sharing information pattern".
\newblock \emph{IEEE Transactions on Automatic Control}, 24\penalty0
  (4):\penalty0 656--657, 1979.
\newblock \doi{10.1109/TAC.1979.1102080}.

\bibitem[Kushner(1971)]{Kushner1971}
H.~J. Kushner.
\newblock \emph{{Introduction to Stochastic Control}}.
\newblock Holt, Rinehart and Winston, 1971.

\bibitem[Leman et~al.(2009)Leman, Xargay, Dullerud, Hovakimyan, and
  Wendel]{Leman2009L1AC}
T.~Leman, E.~Xargay, G.~Dullerud, N.~Hovakimyan, and T.~Wendel.
\newblock L1 adaptive control augmentation system for the x-48b aircraft.
\newblock In \emph{AIAA guidance, navigation, and control conference}, 2009.

\bibitem[Mahajan(2008)]{mahajan2008sequential}
A.~Mahajan.
\newblock \emph{Sequential Decomposition of Sequential Dynamic Teams:
  Applications to Real-Time Communication and Networked Control Systems.}
\newblock PhD thesis, University of Michigan, 2008.

\bibitem[{Mahajan} et~al.(2012){Mahajan}, {Martins}, {Rotkowitz}, and
  {Y{\"u}ksel}]{Mahajan2012}
A.~{Mahajan}, N.~C. {Martins}, M.~C. {Rotkowitz}, and S.~{Y{\"u}ksel}.
\newblock Information structures in optimal decentralized control.
\newblock In \emph{2012 IEEE 51st IEEE Conference on Decision and Control
  (CDC)}, pages 1291--1306, Dec 2012.
\newblock \doi{10.1109/CDC.2012.6425819}.

\bibitem[Malikopoulos(2009)]{Malikopoulos2009b}
A.~A. Malikopoulos.
\newblock {Convergence properties of a computational learning model for unknown
  Markov chains}.
\newblock \emph{J. Dyn. Sys., Meas., Control}, 131\penalty0 (4):\penalty0
  041011--7, 2009.

\bibitem[Malikopoulos(2016)]{Malikopoulos2016c}
A.~A. Malikopoulos.
\newblock A duality framework for stochastic optimal control of complex
  systems.
\newblock \emph{IEEE Transactions on Automatic Control}, 61\penalty0
  (10):\penalty0 2756--2765, 2016.

\bibitem[Malikopoulos(2023)]{Malikopoulos2021}
A.~A. Malikopoulos.
\newblock On team decision problems with nonclassical information structures.
\newblock \emph{IEEE Transactions on Automatic Control}, 2023.

\bibitem[Malikopoulos et~al.(2010)Malikopoulos, Papalambros, and
  Assanis]{Malikopoulos2010a}
A.~A. Malikopoulos, P.~Y. Papalambros, and D.~N. Assanis.
\newblock Online identification and stochastic control for autonomous internal
  combustion engines.
\newblock \emph{Journal of Dynamic Systems, Measurement, and Control},
  132\penalty0 (2):\penalty0 024504--024504, 2010.

\bibitem[McGuire and Radner(1972)]{McGuire1972}
C.~B. McGuire and R.~Radner, editors.
\newblock \emph{Decision and Organization: A Volume in Honor of Jacob Marschak
  (Studies in Mathematical and Managerial Economics)}.
\newblock North-Holland Pub. Co, 1972.

\bibitem[Narendra and Annaswamy(1989)]{Narendra1989StableAS}
K.~Narendra and A.~Annaswamy.
\newblock \emph{Stable Adaptive Systems}.
\newblock Prentice-Hall, Inc., 1989.

\bibitem[Nayyar et~al.(2011)Nayyar, Mahajan, and Teneketzis]{Nayyar2011}
A.~Nayyar, A.~Mahajan, and D.~Teneketzis.
\newblock {Optimal control strategies in delayed sharing information
  structures}.
\newblock \emph{IEEE Transactions on Automatic Control}, 56\penalty0
  (7):\penalty0 1606--1620, 2011.
\newblock ISSN 00189286.
\newblock \doi{10.1109/TAC.2010.2089381}.

\bibitem[Nayyar et~al.(2013)Nayyar, Mahajan, and Teneketzis]{Nayyar2013b}
A.~Nayyar, A.~Mahajan, and D.~Teneketzis.
\newblock {Decentralized stochastic control with partial history sharing: A
  common information approach}.
\newblock \emph{IEEE Transactions on Automatic Control}, 58\penalty0
  (7):\penalty0 1644--1658, 2013.

\bibitem[Ooi et~al.(1997)Ooi, Verbout, Ludwig, and Wornell]{Ooi:1997aa}
J.~M. Ooi, S.~M. Verbout, J.~T. Ludwig, and G.~W. Wornell.
\newblock A separation theorem for periodic sharing information patterns in
  decentralized control.
\newblock \emph{IEEE Transactions on Automatic Control}, 42\penalty0
  (11):\penalty0 1546--1550, 1997.
\newblock \doi{10.1109/9.649699}.

\bibitem[Papadimitriou and Tsitsiklis(1982)]{papadimitriou1982complexity}
C.~H. Papadimitriou and J.~Tsitsiklis.
\newblock On the complexity of designing distributed protocols.
\newblock \emph{Information and Control}, 53\penalty0 (3):\penalty0 211--218,
  1982.

\bibitem[Papadimitriou and Tsitsiklis(1985)]{Papadimitriou:1985aa}
C.~H. Papadimitriou and J.~N. Tsitsiklis.
\newblock Intractable problems in control theory.
\newblock In \emph{1985 24th IEEE Conference on Decision and Control}, pages
  1099--1103, 1985.
\newblock \doi{10.1109/CDC.1985.268670}.

\bibitem[Papadimitriou and
  Tsitsiklis(1987)]{Papadimitriou:1987:CMD:2875343.2875347}
C.~H. Papadimitriou and J.~N. Tsitsiklis.
\newblock The complexity of markov decision processes.
\newblock \emph{Math. Oper. Res.}, 12\penalty0 (3):\penalty0 441--450, Aug.
  1987.
\newblock ISSN 0364-765X.

\bibitem[Recht(2019)]{Recht2018ATO}
B.~Recht.
\newblock A tour of reinforcement learning: The view from continuous control.
\newblock \emph{Annual Review of Control, Robotics, and Autonomous Systems},
  2:\penalty0 253--279, 2019.

\bibitem[Rosolia and Borrelli(2018)]{Rosolia:2018wv}
U.~Rosolia and F.~Borrelli.
\newblock Learning model predictive control for iterative tasks. a data-driven
  control framework.
\newblock \emph{IEEE Transactions on Automatic Control}, 63\penalty0
  (7):\penalty0 1883--1896, 2018.

\bibitem[Sahoo and Vamvoudakis(2020)]{Sahoo:2020tx}
P.~P. Sahoo and K.~G. Vamvoudakis.
\newblock On-off adversarially robust q-learning.
\newblock \emph{IEEE Control Systems Letters}, 4\penalty0 (3):\penalty0
  749--754, 2020.

\bibitem[Sastry and Bodson(1989)]{Sastry1989AdaptiveCS}
S.~Sastry and M.~Bodson.
\newblock \emph{Adaptive Control: Stability, Convergence and Robustness}.
\newblock Prentice-Hall, Inc., 1989.

\bibitem[Sondik(1971)]{Sondik1971}
E.~J. Sondik.
\newblock \emph{The Optimal Control of Partially Observed Markov Processes}.
\newblock PhD thesis, Stanford University, 1971.

\bibitem[Striebel(1965)]{Striebel1965}
C.~Striebel.
\newblock {Sufficient statistics in the optimum control of stochastic systems}.
\newblock \emph{Journal of Mathematical Analysis and Applications}, 12\penalty0
  (3):\penalty0 576--592, 1965.

\bibitem[Subramanian and Mahajan(2019)]{Subramanian2019ApproximateIS}
J.~Subramanian and A.~Mahajan.
\newblock Approximate information state for partially observed systems.
\newblock \emph{2019 IEEE 58th Conference on Decision and Control (CDC)}, pages
  1629--1636, 2019.

\bibitem[Subramanian et~al.(2022)Subramanian, Sinha, Seraj, and
  Mahajan]{Subramanian2020ApproximateIS}
J.~Subramanian, A.~Sinha, R.~Seraj, and A.~Mahajan.
\newblock Approximate information state for approximate planning and
  reinforcement learning in partially observed systems.
\newblock \emph{Journal of Machine Learning Research}, 23:\penalty0 1--83,
  2022.

\bibitem[Sutton and Barto(1998)]{Sutton1998a}
R.~S. Sutton and A.~G. Barto.
\newblock \emph{Reinforcement Learning: An Introduction}.
\newblock Bradford Books, 1998.

\bibitem[Tsitsiklis and Athans(1985)]{tsitsiklis1985complexity}
J.~Tsitsiklis and M.~Athans.
\newblock On the complexity of decentralized decision making and detection
  problems.
\newblock \emph{IEEE Transactions on Automatic Control}, 30\penalty0
  (5):\penalty0 440--446, 1985.

\bibitem[van Schuppen and Villa(2015)]{Schuppen2015}
J.~H. van Schuppen and T.~Villa.
\newblock \emph{Coordination Control of Distributed Systems}.
\newblock Springer, 2015.

\bibitem[Varaiya and Walrand(1978)]{Varaiya:1978aa}
P.~Varaiya and J.~Walrand.
\newblock On delayed sharing patterns.
\newblock \emph{IEEE Transactions on Automatic Control}, 23\penalty0
  (3):\penalty0 443--445, 1978.
\newblock \doi{10.1109/TAC.1978.1101739}.

\bibitem[{Witsenhausen}(1971)]{Witsenhausen1971a}
H.~S. {Witsenhausen}.
\newblock Separation of estimation and control for discrete time systems.
\newblock \emph{Proceedings of the IEEE}, 59\penalty0 (11):\penalty0
  1557--1566, Nov 1971.
\newblock ISSN 0018-9219.

\bibitem[Witsenhausen(1973)]{Witsenhausen1973}
H.~S. Witsenhausen.
\newblock A standard form for sequential stochastic control.
\newblock \emph{Math. Syst. theory}, 7\penalty0 (1):\penalty0 5--11, 1973.

\bibitem[Wu et~al.(2017)Wu, Parvate, Kheterpal, Dickstein, Mehta, Vinitsky, and
  Bayen]{Wu:2017uz}
C.~Wu, K.~Parvate, N.~Kheterpal, L.~Dickstein, A.~Mehta, E.~Vinitsky, and A.~M.
  Bayen.
\newblock Framework for control and deep reinforcement learning in traffic.
\newblock In \emph{2017 IEEE 20th International Conference on Intelligent
  Transportation Systems (ITSC)}, pages 1--8, 2017.
\newblock ISBN 2153-0017.

\bibitem[Wu and Lall(2014)]{wu2014theory}
J.~Wu and S.~Lall.
\newblock A theory of sufficient statistics for teams.
\newblock In \emph{53rd IEEE Conference on Decision and Control}, pages
  2628--2635. IEEE, 2014.

\bibitem[Zhai and Vamvoudakis(2021)]{Zhai:2021wy}
L.~Zhai and K.~G. Vamvoudakis.
\newblock A data-based private learning framework for enhanced security against
  replay attacks in cyber-physical systems.
\newblock \emph{International Journal of Robust and Nonlinear Control},
  31\penalty0 (6):\penalty0 1817--1833, 2021.

\end{thebibliography}

%%%%%%%%%%%%%%%%%%%%%%%%%%%%%%%%%%%%%%%%%%%%%%%%%%%%%%%%%%%%%%%%%%%%%%%%%%%%%%%%
%\balance

\begin{wrapfigure}{l}{0.10\textwidth}
	\centering
	\includegraphics[width=0.13\textwidth]{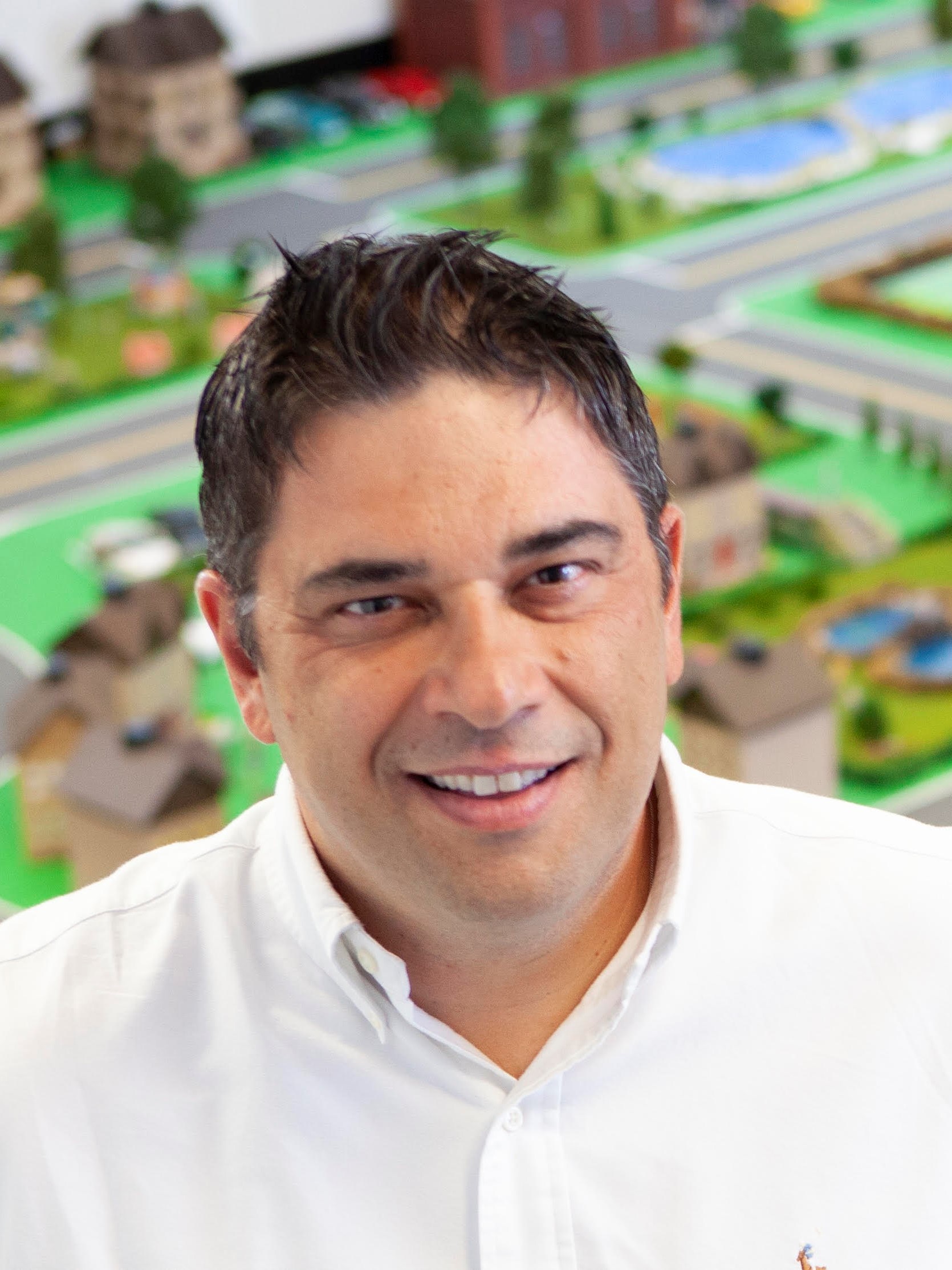}
	%\caption{\label{bio1}This is a figure caption.}
\end{wrapfigure}
\textbf{Andreas A. Malikopoulos}  received the Diploma in mechanical engineering from the National Technical University of Athens, Greece, in 2000. He received M.S. and Ph.D. degrees from the department of mechanical engineering at the University of Michigan, Ann Arbor, Michigan, USA, in 2004 and 2008, respectively. 
He is the Terri Connor Kelly and John Kelly Career Development Associate Professor in the Department of Mechanical Engineering at the University of Delaware, the Director of the Information and Decision Science (IDS) Laboratory, and the Director of the Sociotechnical Systems Center. Prior to these appointments, he was the Deputy Director and the Lead of the Sustainable Mobility Theme of the Urban Dynamics Institute at Oak Ridge National Laboratory, and a Senior Researcher with General Motors Global Research \& Development. His research spans several fields, including analysis, optimization, and control of cyber-physical systems; decentralized systems; stochastic scheduling and resource allocation problems; and learning in complex systems. The emphasis is on applications related to smart cities, emerging mobility systems, and sociotechnical systems. He has been an Associate Editor of the IEEE Transactions on Intelligent Vehicles and IEEE Transactions on Intelligent Transportation Systems from 2017 through 2020. He is currently an Associate Editor of Automatica and IEEE Transactions on Automatic Control. He is a member of SIAM and AAAS. He is also a Senior Member of the IEEE and a Fellow of the ASME.
\end{document}